\documentclass[12pt]{amsart}

\usepackage{graphicx}
\graphicspath{{images/}}
\usepackage{comment}

\usepackage{amssymb,amsmath,amsfonts,latexsym}
\usepackage{bm}
\usepackage[all,cmtip]{xy}
\usepackage{amscd}
\usepackage{mathrsfs} 
\usepackage[shortlabels]{enumitem} 

\newcommand{\bea}{\begin{eqnarray}}
\newcommand{\eea}{\end{eqnarray}}

\newcommand{\clh}{\mathcal{H}}
\newcommand{\clk}{\mathcal{K}}

\newcommand{\clm}{\mathcal{M}}
\newcommand{\cln}{\mathcal{N}}

\newcommand{\z}{\bm{z}}

\newcommand{\K}{\bm{k}}

\newcommand{\bxi}{\bm{\xi}}
\newcommand{\bzeta}{\bm{\zeta}}

\newcommand{\T}{\mathbb{T}}
\newcommand{\N}{\mathbb{N}}
\newcommand{\D}{\mathbb{D}}
\newcommand{\Z}{\mathbb{Z}}
\newcommand{\C}{\mathbb{C}}

\newcommand{\lt}{\left}
\newcommand{\rt}{\right}

\newcommand{\Tscr}{\mathscr{T}}
\newcommand{\Bscr}{\mathscr{B}}
\newcommand{\Sscr}{\mathscr{S}}
\newcommand{\Cscr}{\mathscr{C}}
\newcommand{\Ascr}{\mathscr{A}}
\newcommand{\B}{\mathfrak{B}}

\def\textmatrix#1&#2\\#3&#4\\{\bigl({#1 \atop #3}\ {#2 \atop #4}\bigr)}
\def\dispmatrix#1&#2\\#3&#4\\{\left({#1 \atop #3}\ {#2 \atop #4}\right)}
\newcommand{\be}{\begin{equation}}
	\newcommand{\ee}{\end{equation}}
\newcommand{\ben}{\begin{eqnarray*}}
	\newcommand{\een}{\end{eqnarray*}}

\newcommand{\NI}{\noindent}

\newcommand{\bi}{\begin{itemize}}
	\newcommand{\ei}{\end{itemize}}

\newcommand{\wt}[1]{\widetilde{#1}}
\newcommand{\ol}[1]{\overline{#1}}
\newcommand\la{{\langle }}
\newcommand\ra{{\rangle}}

\theoremstyle{definition}

\newtheorem*{theorem*}{Theorem}
\theoremstyle{plain}

\newtheorem{thm}{Theorem}[section]
\newtheorem{cor}[thm]{Corollary}
\newtheorem{lem}[thm]{Lemma}
\newtheorem{prop}[thm]{Proposition}

\theoremstyle{definition}

\newtheorem{rem}[thm]{Remark}
\newtheorem{ex}[thm]{Example}

\numberwithin{equation}{section}

\let\phi=\varphi

\begin{document}

\title[Toeplitz algebra and Symbol map via Berezin transform]{Toeplitz algebra and Symbol map via Berezin transform on $H^2(\D^n)$}
	
\author[Javed]{Mo Javed}
\address{Indian Institute of Technology Roorkee, Department of Mathematics,
		Roorkee-247 667, Uttarakhand,  India}
\email{javediitr07@gmail.com}

\author[Maji]{Amit Maji\dag}
\address{Indian Institute of Technology Roorkee, Department of Mathematics,
		Roorkee-247 667, Uttarakhand,  India}
\email{amit.maji@ma.iitr.ac.in, amit.iitm07@gmail.com, (\dag{Corresponding author)}}




\begin{abstract}
Let $\mathscr{T}(L^{\infty}(\mathbb{T}))$ be the Toeplitz algebra, that is, the $C^*$-algebra generated by the set $\{T_{\phi} : \phi\in L^{\infty}(\mathbb{T})\}$. Douglas's theorem on \emph{symbol map} states that there exists a $C^*$-algebra homomorphism from $\mathscr{T}(L^{\infty}(\mathbb{T}))$ onto $L^{\infty}(\mathbb{T})$ such that $T_{\phi}\mapsto \phi$ and the kernel of the homomorphism coincides with commutator ideal in $\mathscr{T}(L^{\infty}(\mathbb{T}))$. In this paper, we use the {\textit{Berezin transform}} to study results akin to Douglas's theorem for operators on the Hardy space $H^2(\mathbb{D}^n)$ over the open unit polydisc $\mathbb{D}^n$ for $n\geq 1$. We further obtain a class of bigger $C^*$-algebras than the Toeplitz algebra $\mathscr{T}(L^{\infty}(\mathbb{T}^n))$ for which the analog of \emph{symbol map} still holds true.
\end{abstract}	
\maketitle

\section{Introduction}\label{Sec:Introduction}

One of the major concerns of operator theory and function theory has generally been the study of operators associated with the spaces of holomorphic functions. We often come across linear operators induced by functions (usually called the symbol of the operator), such as Toeplitz operators, Hankel operators, and composition operators; on the other hand, we can investigate functions that are induced by an operator. A great illustration of the relationship between operators and functions is the {\textit{Berezin transform}}.

The {\textit{Berezin transform}}, which was first introduced in \cite{Berezin-Covariant and contravariant} as a tool in quantization, is named after F. A. Berezin. Usually, it involves an open subset $\mathcal{G}$ of $\C^n$ ($n \geq 1$) and a reproducing kernel Hilbert space ${H}(\mathcal{G})$ with a reproducing kernel, say, $K_{\z}$ for 
$\z\in\mathcal{G}$. Then, for a bounded operator $A$ on ${H}(\mathcal{G})$, its {\textit{Berezin transform}}, denoted as $\widetilde{A}$, is a complex-valued function 
on $\mathcal{G}$ defined by
\[
\widetilde{A}(\z)=\langle Ak_{\z}, k_{\z} \rangle_{{H}(\mathcal{G})} \quad (\z \in \mathcal{G}),
\]
where $k_{\z}={K_{\z}}/{\|K_{\z}\|_{{H}(\mathcal{G})}}$ is the normalized reproducing kernel of ${H}(\mathcal{G})$. Thus every bounded linear operator $A$ on ${H}(\mathcal{G})$ induces a bounded function $\widetilde{A}$ on $\mathcal{G}$. Indeed, $\| \widetilde{A} \|_{\infty} \leq \| A\|$. Therefore, the {\textit{Berezin transform}} of an operator yields a crucial information about the operator. Beyond its original purpose, the {\textit{Berezin transform}} has shown usefulness in a number of contexts, including the Hardy space, the Bergman space (\cite{Stf-ALGEBRAIC PROP. OF TOEP. OPER. via BEREZIN}, \cite{Axlr-Berezin transform on the Toeplitz algebra}), the Bargmann-Segal space \cite{Brg_Cbr-TOEPL. OPER. QUANTUM MECH.}. It has also established important relationships with the Bloch space and functions of bounded mean oscillation, and the detailed investigation of these connections can be found in \cite{Zhu-VMO ESV and TOEP. OPER.}. It has also found many applications in many areas of mathematical analysis and mathematical physics (see \cite{Berezin- Quantization}, \cite{Engls-Berezin quantization and reproducing kernels}, \cite{Zhu-Analysis on Fock Spaces}).

Toeplitz operators emerge to be one of the most distinguishable classes of concrete
operators because of their immensely useful and widespread applications including operator theory, function theory, operator algebras, and associated disciplines. 
Furthermore, it is preferable to think of these operators as elements within operator algebras, like the Toeplitz algebra $\Tscr(L^{\infty}(\T))$), that is, the $C^*$-algebra generated by the set $\{T_{\phi} : \phi\in L^{\infty}(\T)\}$ of Toeplitz operators on the Hardy space $H^2(\D)$ (see \cite{{Coburn-$C^*$-algebra gtd by an isometry}}, \cite{Dgl-BANACH ALGEBRA TECH. IN OPER. TH.}). The \emph{symbol map}, which was initially discovered and utilized by Douglas, proves to be an important tool in the study of the Toeplitz algebra. The classical result of Douglas's \emph{symbol map} \cite[chapter 7]{Dgl-BANACH ALGEBRA TECH. IN OPER. TH.} is as follows:
\begin{thm}\label{thm:Douglas's classic result}
Let $\Tscr(L^{\infty}(\T))$ denote the $C^*$-algebra generated by the set
$\{T_{\phi} : \phi\in L^{\infty}(\T)\}$. Then there exists a $C^*$-homomorphism 
\[
\sigma : \Tscr(L^{\infty}(\T)) \to L^{\infty}(\T)
\]
such that kernel of $\sigma$ is the commutator ideal, that is, the ideal in 
$\Tscr(L^{\infty}(\T))$ generated by the commutators $[A, B]=AB-BA$; $A,B\in \Tscr(L^{\infty}(\T))$, and $\sigma(T_{\phi})=\phi$ for all $\phi\in L^{\infty}(\T)$.  
\end{thm}
The $C^*$-homomorphism $\sigma$ is called the \emph{symbol map}. However, Douglas's proof of Theorem \ref{thm:Douglas's classic result} was mere existential and much abstract in nature until 1982, when Barr\'{i}a and Halmos in \cite{Bar_Hlms-ASYMPTOTIC TOEPLITZ OPERATORS} presented a constructive proof of Theorem \ref{thm:Douglas's classic result}. More specifically, Barr\'{i}a and Halmos obtained the symbol of an operator $A\in\Tscr(L^{\infty}(\T))$ as the symbol of the Toeplitz operator in the SOT-limit of the sequence $\{T_z^{*m}AT_z^m\}_{m=1}^{\infty}$, where $T_{z}$ denotes the unilateral shift on the Hardy space $H^2(\D)$. However, due to the nuances of limits in the strong operator topology, it is still very difficult to find the symbol of any operator $A$ in $\Tscr(L^{\infty}(\T))$. The {\textit{Berezin transform}} is a useful tool in this way to obtain the symbol of an operator $A\in\Tscr(L^{\infty}(\T))$ more concretely. Indeed, the {\textit{Berezin transform}} of a Toeplitz operator on Hardy space reduces to the Poisson integral of the symbol of that Toeplitz operator (see Theorem \ref{thm:radial limit of Toeplitz operaor} below). In particular, Engli\v{s} \cite{Engls-SYMBOL MAP VIA BEREZIN TRAN. ON H2} developed a technique using the {\textit{Berezin transform}} in which he obtained the symbol of an operator $A\in\Tscr(L^{\infty}(\T))$ as the radial boundary function of the {\textit{Berezin transform}} of $A$. The key observation in his proof is that 
\[
T_{\phi}K_{\xi}=\phi(\xi)K_{\xi},
\]
for the densely defined Toeplitz operator $T_{\phi}$ with symbol $\phi\in H^2(\D)^{\perp}$ 
and $K_{\xi}$ is the reproducing kernel of $H^2(\D)$. It uses the fact that 
$\overline{\phi}\in H^2(\D)$ whenever $\phi\in H^2(\D)^{\perp}$. However, this technique fails in the case $\D^n$ for $n>1$, for an example: if we take $\phi(z_1, z_2)=\overline{z_1}z_2$ then both $\phi$ and $\overline{\phi}$ are in $H^2(\D^2)^{\perp}$, and $T_{\phi}K_{\bxi}\neq \phi(\bxi)K_{\bxi}$ for $\bxi\in \D^2$, where $K_{\bxi}(\z)=\dfrac{1}{(1-\overline{\xi_1}z_1)(1-\overline{\xi_2} z_2)}$ is a reproducing kernel of $H^2(\D^2)$. Thus the natural question arises:
 
\vspace{0.1 cm}
\textsf{Question: Is the Douglas's symbol map valid in the setting of the open unit polydisc $\D^n$ for $n >1$?}
\vspace{0.1 cm}	

In this paper we present a comprehensive study of symbols of operators on $H^2(\D^n)$.	
The main contribution of this article is to address the above question with the help of the {\textit{Berezin transform}}. More precisely, we obtain the following result
(see Theorem \ref{thm:douglas theorem for polydisk} below):

\begin{theorem*}
Let $T\in \Tscr(L^{\infty}(\T^n))$, the Toeplitz algebra of operators on $H^2(\D^n)$ for $n \geq 1$. Then $\widetilde{T}\to \phi$ radially for some $\phi\in L^{\infty}(\T^n)$. Moreover, the map 
\[
\Sigma : \Tscr(L^{\infty}(\T^n)) \to L^{\infty}(\T^n), \qquad T \mapsto \phi
\]
is a $C^*$-algebra homomorphism with ker$(\Sigma)=\Sscr$, the \emph{semi-commutator ideal}	in $\Tscr(L^{\infty}(\T^n))$ and $\Sigma(T_{\phi})=\phi$ for any Toeplitz operator $T_{\phi}$.
\end{theorem*}

The rest of the paper is outlined as follows: Section \ref{Sec:Preliminaries} is devoted to establishing the necessary background, and along the way, we fix some notations and state some properties of Poison integral and \textit{Berezin transform}. In Section 3, we prove an important result concerning the radial limits, which serves as a key ingredient in the proof of the \emph{symbol map}. Section 4 concerns with the main result, that is, the \emph{symbol map} (see Theorem \ref{thm:douglas theorem for polydisk}), and we depict through an example that Theorem \ref{thm:douglas theorem for polydisk} can not be extended to the whole of $\B(H^2(\D^n))$, the $C^*$-algebra of all bounded operators on $H^2(\D^n)$. In section 5, we present a family of larger $C^*$-algebras than the Toeplitz algebra 
$\Tscr(L^{\infty}(\T^n))$ for which the analog of symbol map still holds true. In section 6, we conclude with some open questions.

\section{Preliminaries}\label{Sec:Preliminaries}

Let $\Z_{+}$ denote the set of all non-negative integers, $\D=\{z\in\C : |z|<1\}$ be the open unit disc in the complex plane $\C$, $\T=\{z\in\C : |z|=1\}$ the unit circle in $\C$, and $\B(\clh)$ denote the set of all bounded linear operators on a Hilbert space $\clh$. For $n \geq 1$, the Cartesian product of $n$ copies of $\D$, denoted by $\D^n$, is the unit polydisc in $\C^n$ and the $n$-torus $\T^n$ denotes the distinguished boundary of $\D^n$. Then $H^2(\D^n)$ denotes as the Hilbert space of all analytic functions $f$ on $\D^n$ for which
\[
\|f\|_{H^2(\D^n)}:=\underset{0\leq r<1}{\text{sup}}\left(\int_{\T^n}\left|f(r\bzeta)\right|^2\,d\mathfrak{m}_n(\bzeta)\right)^{1/2} < \infty,
\] 
where $\mathfrak{m}_n$ stands for the normalised Lebesgue measure on $\T^n$. We shall write bold faced $\bxi$ and $\bzeta$, respectively as vectors $\bxi=(\xi_1, \dots,\xi_n)\in \D^n$ and $\bzeta=(\zeta_1, \dots,\zeta_n)\in\T^n$ for $n >1$. For $1\leq p<\infty$, $L^{p}(\T^n):=L^{p}(\T^n, d\mathfrak{m}_n)$ is the Banach space of equivalence classes of Lebesgue complex measurable functions $f$ such that $|f|^p$ is integrable on $\T^n$ with respect to $\mathfrak{m}_n$ and $p=2$ gives the Hilbert space $L^2(\T^n)$ of square-integrable equivalence classes of functions with respect to $\mathfrak{m}_{n}$. Also $L^{\infty}(\T^n)$ denotes the commutative $C^*$-algebra of essentially bounded complex measurable functions on $\T^n$ with respect to $\mathfrak{m}_n$, and $H^{\infty}(\D^n)$ stands for the space of bounded holomorphic functions on $\D^n$. Let $\{e_{\K}\}_{\K\in\Z^n}$ denote the standard orthonormal basis for $L^2(\T^n)$, where $e_{\K}(\bzeta)=\zeta_{1}^{k_1}\dots\zeta_{n}^{k_n}$ for $\K=(k_1, \dots,k_n)\in\Z^n$. 
Consider the closed subspace $H^2(\T^n)$ of $L^2(\T^n)$ defined by 
\[
H^2(\T^n):=\{f\in L^2(\T^n) : \la f, e_{\K} \ra_{L^2(\T^n)} = 0\; \mbox{for all~} \K \in\Z^n\setminus\Z_{+}^n \}.
\]
We then identify the Hardy space $H^2(\D^n)$ (via radial limits) with $H^2(\T^n)$ by the unitary map 
\[
H^2(\D^n) \ni \widetilde{f}=\sum_{\K\in \Z_{+}^n}a_{\K}\z^{\K} \mapsto \sum_{\K\in \Z_{+}^n}a_{\K}e_{\K} = f \in H^2(\T^n).
\] 
Now onwards, without any ambiguity, we shall interchangeably write $H^2(\D^n)$ (respectively, its elements $\widetilde{f}$) in place of $H^2(\T^n)$ (respectively, its elements $f$) and vice-versa. Furthermore, for $\widetilde{f}, \widetilde{g}\in H^2(\D^n)$  and $f, g\in H^2(\T^n)$, we have
\[
\langle\widetilde{f},\widetilde{g}\rangle_{H^2(\D^n)}=\langle f, g \rangle_{L^2(\T^n)}.
\]

For $\phi\in L^{\infty}(\T^n)$, the Toeplitz operator $T_{\phi}$ is a bounded operator on $H^2(\D^n)$ given by 
\[
T_{\phi}f=P_{H^2(\D^n)}(\phi f) \qquad (f\in H^2(\D^n)),
\]
and the Hankel operator $H_{\phi}$ with symbol $\phi\in L^{\infty}(\T^n)$ is a bounded operator from $H^2(\D^n)$ to ${H^2(\D^n)^\perp}$ defined by
\[
H_{\phi}f=P_{H^2(\D^n)^\perp}(\phi f) \qquad (f\in H^2(\D^n)),
\]	
where $P_{H^2(\D^n)}$ is the orthogonal projection of $L^2(\T^n)$ onto $H^2(\D^n)$ and
$P_{H^2(\D^n)^\perp} = I_{L^2(\T^n)} - P_{H^2(\D^n)}$. Clearly, the adjoint operator $T_{\phi}^* = T_{\bar{\phi}}$, where ${\bar{\phi}}$ is the complex conjugate of $\phi$. 
It is easy to check that $[T_{\phi}, T_{\psi}) = H_{\overline{\phi}}^{*}H_{\psi}$
for $\phi, \psi \in L^{\infty}(\T^n)$.

We frequently use the identification of the Hardy space over the polydisc $H^2(\D^n)$ with $H^2(\D)\otimes \cdots \otimes H^2(\D)$, the $n$-fold ($n>1$) tensor product of Hilbert space $H^2(\D)$  via the canonical unitary map $z_1^{m_1} \dots z_n^{m_n} \mapsto z^{m_1} \otimes \cdots \otimes z^{m_n}$, where $m_i \in \Z_{+}$. Thus, we shall often write $(f_1\otimes \cdots \otimes f_n)=f_1 \cdots f_n$, where $f_j\in H^2(\D)$. The Hardy space over the polydisc $H^2(\D^n)$ is a reproducing kernel Hilbert space on $\D^n$ for $n \geq 1$. For $\bxi=(\xi_1, \dots, \xi_n)\in \D^n$, the (Szeg\"{o}) kernel $K_{\bxi}$ of $H^2(\D^n)$ is defined as
\[
K_{\bxi}(\bm{z})=\prod_{j=1}^{n}\frac{1}{(1-\overline{\xi_j}z_j)}, \qquad (\bm{z} \in \D^n)
\]
and it reproduces the value of each $f \in H^2(\D^n)$, that means $f(\bxi) =\langle f, K_{\bxi} \rangle $. The normalised reproducing kernel is denoted by
\[
k_{\bxi}(\bm{z}) =\frac{K_{\bxi}(\bm{z})}{\|K_{\bxi}\|} =\frac{\sqrt{(1-|\xi_1|^2) \cdots (1-|\xi_n|^2)}} {(1-\overline{\xi_1}z_1) \cdots (1-\overline{\xi_n}z_n)},
\qquad (\bm{z} \in \D^n).
\] 
In particular, the reproducing kernel $K_{\xi}$ ($\xi\in\D$) of the Hardy space over the disc $H^2(\D)$ is given by
\[
K_{\xi}(z)=\frac{1}{1-\overline{\xi}z} \quad (z \in \D),
\]
and $k_{\xi}(z)= \frac{\sqrt{1-|\xi|^2}}{1-\overline{\xi}z }$ is the normalised reproducing kernel of $H^2(\D)$.

Let $\bxi=(\xi_1, \dots, \xi_n)\in \D^n, \bzeta=(\zeta_1, \dots, \zeta_n)\in \T^n$,
$\xi_j=r_je^{i\theta_j}$ and $\zeta_j=e^{it_j}$ where $0\leq r_j<1$ and $\theta_j, t_j\in [0, 2\pi)$ for all $1\leq j\leq n$. The Poisson kernel $P(\bxi, \bzeta)$ for the unit polydisc $\D^n$ is defined as
\begin{align*}
P(\bxi, \bzeta)
&= \prod_{j=1}^{n}\frac{1-|\xi_j|^2}{|1-\xi_j \overline{\zeta_j}|^2}\\
&= P_{r_1}(\theta_1-t_1) \cdots P_{r_n}(\theta_n-t_n),
\end{align*}
where $P_{r}(\theta)=\frac{1-r^2}{1-2r\cos\theta+r^2}$ is the Poisson kernel for the unit disc $\D$. For a complex Borel measure $\mu$ on $\T^n$, its Poisson integral, denoted by $P[d\mu]$, is the integral given by 
\[
P[d\mu](\bxi)=\int_{\T^n}P(\bxi, \bzeta)\,d\mu(\bzeta) \qquad (\bxi\in \D^n).
\]

\NI
For $\phi\in L^{1}(\T^n)$, its harmonic extension $\wt{\phi}$ into $\D^n$ is defined by the Poisson integral formula
\[
\wt{\phi}(\bxi) =P[\phi d\mathfrak{m}_{n}](\bxi)
=\int_{\T^n}P(\bxi, \bzeta)\phi(\bzeta)\,d\mathfrak{m}_n(\bzeta) \qquad (\bxi\in\D^n).
\]
	
\NI	
Let $\Phi: \D^n \to \C$ and $\phi: \T^n \to \C$ be any two complex functions. Define $\hat{\Phi}$ as
\[
\hat{\Phi}(\bzeta)=\lim_{r\to 1^{-}}\Phi(r\bzeta)
\]
for all $\bzeta\in\T^n$ at which the (radial) limit on the right exists. We say that `$\Phi$ approaches $\phi$ radially', denoted by `$\Phi \to \phi$ \emph{radially}' or `$\Phi(\bxi) \to \phi$ \emph{radially}' if 
\[
\hat{\Phi}(\bzeta)=\phi(\bzeta) \quad \text{for a.e.~ } \bzeta\in\T^n.
\]

\begin{thm}\label{thm: radial limit theorem}
\textnormal{(See Theorem 2.3.1 in \cite{Rudn-FUNCTION THEORY IN POLYDISCS})} 
Let $\phi\in L^1(\T^n)$, and $\nu$ be a measure on $\T^n$ such that $\nu$ is singular 
with respect to $\mathfrak{m}_n$. If $\Phi=P[\phi d\mathfrak{m}_n + d\nu]$, then 
\[
\Phi \to \phi \; \text{ radially }.
\]
In particular, if $\nu=0$, then $\wt{\phi} \to \phi$ radially.
\end{thm}

The \textit{Berezin transform} of $T\in \B(H^2(\D^n))$ is a complex function on $\D^n$ defined by
\[
\widetilde{T}(\bxi)=\langle Tk_{\bxi}, k_{\bxi} \rangle \quad \text{ for all } \bxi\in \D^n.
\]
Here $k_{\bxi}$ is the normalised reproducing kernel of $H^2(\D^n)$.

The following notions are important and used throughout the paper:
 
\begin{enumerate}
\item For $\phi$, $\psi\in L^{\infty}(\T^n)$, the \emph{semicommutator} is defined as
\[
[T_{\phi}, T_{\psi}) = T_{\phi\psi} - T_{\phi}T_{\psi}.
\]
\item $\Tscr(L^{\infty}(\T^n)):=C^*$-algebra generated by the set $\{T_{\phi} : \phi\in L^{\infty}(\T^n)\}$, called the Toeplitz algebra of operators on $H^2(\D^n)$.
\item $\Cscr:=$ closed ideal in $\Tscr(L^{\infty}(\T^n))$ generated by the set $\{[A, B]=AB-BA : A, B\in \Tscr(L^{\infty}(\T^n)) \},$ called the \emph{commutator ideal}.
\item $\Sscr:=$ closed ideal in $\Tscr(L^{\infty}(\T^n))$ generated by the set
$\{[T_{\phi}, T_{\psi}) : \phi, \psi\in L^{\infty}(\T^n)\},$ called the 
\emph{semicommutator ideal}.
\end{enumerate}

\section{On Radial Limits}
	
In this section we shall discuss about some key results on radial limits which will be used in the next section. Before moving into our main result, we will fix some notations first.

Suppose $H^2(\D)^1:=H^2(\D)$ and $H^2(\D)^{-1}:=H^2(\D)^{\perp}$.
For $\alpha=(\alpha_1, \dots, \alpha_n)\in \{-1, 1\}^n$ and $\K=(k_1, \dots, k_n)\in 
\Z^n$, we write 
\[
H^2(\D^n)^{\alpha} := \bigvee\{e_{\K} : k_j<0 \text{ if } \alpha_j=-1 \text{ and } k_j\geq 0 \text{ if } \alpha_j=1 \text{ for } 1\leq j\leq n \}.
\]
It is easy to check that $H^2(\D^n)^{\alpha}$ is a closed subspace of $L^2(\T^n)$ for all $\alpha\in \{-1,1\}^n$. We identify $H^2(\D^n)^{\alpha}$ with the $n$-fold tensor product of Hibert spaces $H^2(\D)^{\alpha_j}$ for $1\leq j\leq n$ by the unitary map $e_{\K} \mapsto e_{k_1}\otimes\cdots\otimes e_{k_n}$ for all $k_j<0$ (provided $\alpha_j=-1$) and $k_j\geq 0$ (provided $\alpha_j=1$), where
\[
(e_{k_1}\otimes\cdots\otimes e_{k_n})(\bzeta) = \zeta_1^{k_1} \dots \zeta_n^{k_n},
\]
for all $\bzeta\in \T^n$, $k_j<0$ (provided $\alpha_j=-1$) and $k_j\geq 0$ (provided $\alpha_j=1$). We shall simply write $\otimes_{j=1}^{n}f_j=\prod_{j=1}^{n}f_j$ for $f_j\in H^2(\D)^{\alpha_j}$, and also write bold faced $\bm{1}=(1,\dots, 1)\in \{-1, 1\}^n$.

The following lemma is an easy consequence of the Poisson integral:
		
\begin{lem} \label{lemma:radial limit lemma}
Let $f, g\in L^2(\T^n)$. Then 
\[
\la fk_{\bxi}, gk_{\bxi} \ra_{L^2(\T^n)} = (\wt{f\overline{g}})(\bxi), \qquad (\bxi\in \D^n).
\]
Moreover, $\la fk_{\bxi}, gk_{\bxi} \ra_{L^2(\T^n)} \rightarrow f\overline{g}$ radially.
\end{lem}

\begin{proof}
Suppose $f, g\in L^2(\T^n)$. Now
\begin{align} 
\la fk_{\bxi}, gk_{\bxi} \ra_{L^2(\T^n)} 
&= \int_{\T^n}(fk_{\bxi})(\bzeta)(\overline{gk_{\bxi}})(\bzeta)\,d\mathfrak{m}_n(\bzeta) \notag \\
&= \int_{\T^n}(f\overline{g})(\bzeta)|k_{\bxi}(\bzeta)|^2\,d\mathfrak{m}_n(\bzeta). \label{eq:radial limit lemma}
\end{align}
Note that 
\[
|k_{\bxi}(\bzeta)|^2
= \frac{(1-|\xi_1|^2) \cdots (1-|\xi_n|^2)}{\left|1-\overline{\xi_1}\zeta_1\right|^2 \cdots \left|1-\overline{\xi_n}\zeta_n\right|^2}=P(\bxi,\bzeta),
\]
where $\bxi=(\xi_1, \dots, \xi_n)\in \D^n$ and $\bzeta=(\zeta_1, \dots, \zeta_n)\in \T^n$. Therefore, the equation \eqref{eq:radial limit lemma} yields
\[
\la fk_{\bxi}, gk_{\bxi} \ra_{L^2(\T^n)} = \wt{(f\overline{g})}(\bxi).
\]
Since $f, g\in L^2(\T^n)$, $f\overline{g}\in L^1(\T^n)$. Hence, Theorem \ref{thm: radial limit theorem} infers that $\langle fk_{\bxi}, gk_{\bxi} \rangle_{L^2(\T^n)} \to f\overline{g}$ radially.
\end{proof}	

The following theorem establishes the fact that the Berezin transform of $T_{\phi}$ 
is the harmonic extension of $\phi$:

\begin{thm} \label{thm:radial limit of Toeplitz operaor}
For $\phi\in L^{\infty}(\T^n)$, $\widetilde{T_{\phi}}(\bxi)=\widetilde{\phi}(\bxi)$ for all $\bxi \in \D^n$. Furthermore, $\widetilde{T_{\phi}} \to \phi$ radially.
\end{thm}

\begin{proof}
Let $\bxi=(\xi_1, \dots, \xi_n)\in \D^n$. Then
\begin{align*} 
\widetilde{T_{\phi}}(\bxi) 
= \langle T_{\phi}k_{\bxi}, k_{\bxi}\rangle 
= \langle P_{H^2(\D^n)}(\phi k_{\bxi}), k_{\bxi} \rangle 
= \langle \phi k_{\bxi}, k_{\bxi} \rangle_{L^2(\T^n)} 
= \widetilde{\phi}(\bxi), 
\end{align*}
where the last equality follows from Lemma \ref{lemma:radial limit lemma}. 

Now Theorem \ref{thm: radial limit theorem} gives $\widetilde{T_{\phi}} \to \phi$ radially.
\end{proof}

\begin{lem} \label{lemma:unit disk lemma}
If $x\in H^2(\D)^{\perp}$, then $P_{H^2(\D)}(xk_{\xi})=\widetilde{x}(\xi)k_{\xi}$ 
for all $\xi \in \D$.
\end{lem}

\begin{proof}
Let $z \in \D$. Then
\begin{align*}
\langle P_{H^2(\D)}(xK_{\xi}), K_{z} \rangle_{H^2(\D)} 
&= \langle xK_{\xi}, K_{z} \rangle_{L^2(\T)}  \\
&= \langle K_{\xi}, \overline{x}K_{z} \rangle_{H^2(\D)}.
\end{align*}
Since $x\in H^2(\D)^{\perp}$, $\overline{x}\in H^2(\D)$, and hence 
$\overline{x}K_{z} \in H^2(\D)$. Thus, by reproducing property
\begin{align*}
\langle P_{H^2(\D)}(xK_{\xi}), K_{z} \rangle_{H^2(\D)} 
&= \overline{\widetilde{\overline{x}}(\xi)K_{z}(\xi)} \\
&= \widetilde{x}(\xi)\langle K_{\xi}, K_{z} \rangle_{H^2(\D)}  \\
&= \langle \widetilde{x}(\xi)K_{\xi}, K_{z} \rangle_{H^2(\D)}  \qquad (\text{for all } z \in \D).
\end{align*}
Therefore, $P_{H^2(\D)}(xk_{\xi})=\widetilde{x}(\xi)k_{\xi}$.
\end{proof}

We now state the main result of this section.
 
\begin{thm} \label{thm:my theorem}
Let $f\in H^2(\D^n)^{\alpha}$ and $g\in H^2(\D^n)^{\beta}$ for $\alpha=(\alpha_1,\dots, \alpha_n), \beta=(\beta_1,\dots, \beta_n)\in \{-1, 1\}^{n}$. Then 
\[
\left\langle P_{H^2(\D^n)}(fk_{\bxi}), P_{H^2(\D^n)}(gk_{\bxi}) \right\rangle
\to f\overline{g}\; \text{ radially},
\]
that is,
\[
\left\langle P_{H^2(\D^n)}\left(fk_{(re^{i\theta_1},\dots, re^{i\theta_n})}\right), 
P_{H^2(\D^n)}\left(gk_{(re^{i\theta_1},\dots, re^{i\theta_n})}\right) \right\rangle
\to f\left(e^{i\theta_1},\dots, e^{i\theta_n}\right)
\overline{g}\left(e^{i\theta_1},\dots, e^{i\theta_n}\right)
\]
as $r \to 1^{-}$ a.e.\ $(e^{i\theta_1}, \dots, e^{i\theta_n}) \in \T^n$.
\end{thm}

\begin{proof}
Suppose that $f\in H^2(\D^n)^{\alpha}$ and $g\in H^2(\D^n)^{\beta}$ for $\alpha=(\alpha_1,\dots, \alpha_n), \beta=(\beta_1,\dots, \beta_n)\in \{-1, 1\}^{n}$.
If either $\alpha=\bm{1}$ or $\beta=\bm{1}$, then
\[
\langle P_{H^2(\D^n)}(fk_{\bxi}), P_{H^2(\D^n)}(gk_{\bxi}) \rangle 
= \langle fk_{\bxi}, gk_{\bxi} \rangle,
\]
and hence $\langle P_{H^2(\D^n)}(fk_{\bxi}), P_{H^2(\D^n)}(gk_{\bxi}) \rangle
\to f\overline{g}$ radially by using Lemma \ref{lemma:radial limit lemma}.

We now assume that $\alpha\neq \bm{1}$ and $\beta\neq \bm{1}$. Firstly, consider 
\[
I=\{j : \alpha_j=-1\} \quad \mbox{and} \quad J=\{j : \beta_j=-1\}.
\]
Let $f=f_1 \otimes\cdots\otimes f_n$ and $g=g_1\otimes \cdots \otimes g_n$, where $f_j\in H^2(\D)^{\alpha_j}$ and $g_j\in H^2(\D)^{\beta_j}$ for all $1\leq j \leq n$. Then
\begin{align*}
\langle P_{H^2(\D^n)}(fK_{\bxi}), K_{\z} \rangle
&= \langle fK_{\bxi}, K_{\z} \rangle_{L^2(\T^n)} \\
&= \langle (f_1 \otimes \cdots \otimes f_n)(K_{\xi_1}\otimes \dots \otimes K_{\xi_n}), K_{z_1}\otimes \cdots \otimes K_{z_n} \rangle_{L^2(\T^n)} \\
&= \langle f_1K_{\xi_1}\otimes \cdots \otimes f_nK_{\xi_n}, K_{z_1}\otimes \cdots \otimes K_{z_n} \rangle_{L^2(\T^n)} \\
&= \prod_{j=1}^{n}\langle f_jK_{\xi_j}, K_{z_j} \rangle_{L^2(\T)} \\
&= \prod_{j=1}^{n}\langle P_{H^2(\D)}(f_jK_{\xi_j}), K_{z_j} \rangle_{H^2(\D)} \\
&= \left(\prod_{j\in I}\widetilde{f_j}(\xi_j)\right)\left(\prod_{j=1}^{n}\langle P_{H^2(\D)}(f_j'K_{\xi_j}), K_{z_j} \rangle_{H^2(\D)}\right), 
\end{align*}
where 
$f_j'=
\begin{cases}
f_j & \text{ if } j\notin I;\\
1 & \text{ if } j\in I.
\end{cases}$, and the last equality follows from Lemma \ref{lemma:unit disk lemma}. Therefore,
\begin{align*}
\langle P_{H^2(\D^n)}(fK_{\bxi}), K_{\z} \rangle
&= \left(\prod_{j\in I}\widetilde{f_j}(\xi_j)\right)\left(\prod_{j=1}^{n}\langle f_j'K_{\xi_j}, K_{z_j} \rangle_{H^2(\D)}\right) \\
&= \left(\prod_{j\in I}\widetilde{f_j}(\xi_j)\right) \left\langle \bigotimes_{j=1}^{n}f_j'K_{\xi_j}, \bigotimes_{j=1}^{n}K_{z_j} \right\rangle \\
&= \left\langle \left(\prod_{j\in I}\widetilde{f_j}(\xi_j)\right)\left( \bigotimes_{j=1}^{n}f_j'K_{\xi_j} \right), K_{\z} \right\rangle,
\end{align*}
for all $\z\in \D^n$. Hence 
\[
P_{H^2(\D^n)}(fk_{\bxi})=\left(\prod_{j\in I}\widetilde{f_j}(\xi_j)\right)\left( \bigotimes_{j=1}^{n}f_j'k_{\xi_j} \right), ~\text{where }
f_j'=
\begin{cases}
f_j & \text{ if } j\notin I;\\
1 & \text{ if } j\in I.
\end{cases}
\]
Similarly, we have
\[
P_{H^2(\D^n)}(gk_{\bxi})=\left(\prod_{j\in J}\widetilde{g_j}(\xi_j)\right)\left( \bigotimes_{j=1}^{n}g_j'k_{\xi_j} \right),~ \text{where }
g_j'=
\begin{cases}
g_j & \text{ if } j\notin J;\\
1 & \text{ if } j\in J.
\end{cases}
\]
Now
\begin{align*}
\langle P_{H^2(\D^n)}(fk_{\bxi}), P_{H^2(\D^n)}(gk_{\bxi}) \rangle 
&= \left\langle \left(\prod_{j\in I}\widetilde{f_j}(\xi_j)\right)\left( \bigotimes_{j=1}^{n}f_j'k_{\xi_j} \right), \left(\prod_{j\in J}\widetilde{g_j}(\xi_j)\right)\left( \bigotimes_{j=1}^{n}g_j'k_{\xi_j} \right)\right\rangle \\
&= \left(\prod_{j\in I}\widetilde{f_j}(\xi_j)\right)\overline{\left(\prod_{j\in J}\widetilde{g_j}(\xi_j)\right)}\left\langle \left( \bigotimes_{j=1}^{n}f_j'k_{\xi_j} \right), \left( \bigotimes_{j=1}^{n}g_j'k_{\xi_j} \right) \right\rangle \\
&= \left(\prod_{j\in I}\widetilde{f_j}(\xi_j)\right)\overline{\left(\prod_{j\in J}\widetilde{g_j}(\xi_j)\right)} \prod_{j=1}^{n}\langle f_j'k_{\xi_j}, g_j'k_{\xi_j} 
\rangle \\
&= \left(\prod_{j\in I}\widetilde{f_j}(\xi_j)\right)\overline{\left(\prod_{j\in J}\widetilde{g_j}(\xi_j)\right)} \prod_{j=1}^{n}\widetilde{(f_j'\overline{g_j'})}(\xi_j) \tag{By Lemma \ref{lemma:radial limit lemma}} \\
\end{align*}
\begin{align*}
&\rightarrow \left(\prod_{j\in I}f_j\right)\overline{\left(\prod_{j\in J}g_j\right)}\left(\prod_{j=1}^{n}f_j'\overline{g_j'}\right)\tag{By Theorem \ref{thm: radial limit theorem}} \\
&= \left(\prod_{j\in I}f_j\right)\overline{\left(\prod_{j\in J}g_j\right)}\left(\prod_{j=1}^{n}f_j'\right)\left(\prod_{j=1}^{n}\overline{g_j'}\right) \\
&= \left(\prod_{j\in I}f_j\right)\left(\prod_{j\in J}\overline{g_j}\right)\left(\prod_{j\in I^c}f_j\right)\left(\prod_{j\in J^c}\overline{g_j}\right) \\
&= \left(\prod_{j=1}^{n}f_j\right)\left(\prod_{j=1}^{n}\overline{g_j}\right) \\
&= f\overline{g}.
\end{align*}
This proves that
\[
\langle P_{H^2(\D^n)}(fk_{\bxi}), P_{H^2(\D^n)}(gk_{\bxi}) \rangle \rightarrow f\overline{g} ~ \text{ radially}. 
\]

Suppose that $f$ is a finite sum of elements of the form $f_1\otimes \cdots \otimes f_n$, where $f_j\in H^2(\D)^{\alpha_j}$, and $g$ is a finite sum of elements of the form $g_1\otimes \cdots \otimes g_n$, where $g_j\in H^2(\D)^{\beta_j}$ 
for all $1\leq j \leq n$. That is,
\[
f=\sum_{p=1}^{l}(f_{1p}\otimes \cdots \otimes f_{np})\; \text{ and }\; 
g=\sum_{q=1}^{l'}(g_{1q}\otimes  \cdots \otimes g_{nq}),
\]
where $f_{jp}\in H^2(\D)^{\alpha_j}$ and $g_{jq}\in H^2(\D)^{\beta_j}$ for all $1\leq j\leq n$. Therefore, 

$\langle P_{H^2(\D^n)}(fk_{\bxi}), P_{H^2(\D^n)}(gk_{\bxi}) \rangle$
\begin{align*}
&= \left\langle P_{H^2(\D^n)}\left(\left(\sum_{p=1}^{l}(f_{1p}\otimes \cdots \otimes f_{np})\right)k_{\bxi}\right), P_{H^2(\D^n)}\left(\left(\sum_{q=1}^{l'}(g_{1q}\otimes \cdots \otimes g_{nq})\right)k_{\bxi}\right) \right\rangle \\
&= \left\langle \sum_{p=1}^{l}P_{H^2(\D^n)}((f_{1p}\otimes \cdots \otimes f_{np})k_{\bxi}), \sum_{q=1}^{l'}P_{H^2(\D^n)}((g_{1q}\otimes \cdots \otimes g_{nq})k_{\bxi})\right\rangle \\
&= \sum_{p=1}^{l}\sum_{q=1}^{l'} \langle P_{H^2(\D^n)}((f_{1p}\otimes  \cdots \otimes f_{np})k_{\bxi}), P_{H^2(\D^n)}((g_{1q}\otimes \cdots \otimes g_{nq})k_{\bxi})\rangle.
\end{align*}
From the above, we obtain that
\begin{align*}
\langle P_{H^2(\D^n)}(fk_{\bxi}), P_{H^2(\D^n)}(gk_{\bxi}) \rangle
&\to \sum_{p=1}^{l}\sum_{q=1}^{l'}
(f_{1p}\otimes \cdots \otimes f_{np}) \overline{(g_{1q}\otimes \cdots \otimes g_{nq})} 
= f\overline{g} ~ \text{ radially. }
\end{align*}

Finally, assume that $f\in H^2(\D)^{\alpha_1}\otimes \cdots\otimes H^2(\D)^{\alpha_n}$ 
and $g\in H^2(\D)^{\beta_1}\otimes \cdots \otimes H^2(\D)^{\beta_n}$. Then we shall show that 
\[
\langle P_{H^2(\D^n)}(fk_{\bxi}), P_{H^2(\D^n)}(gk_{\bxi}) \rangle \to f\overline{g}\; \text{ radially. }
\] 
Now there exist sequences $(f^{(m)})_{m=1}^{\infty}$ and $(g^{(m)})_{m=1}^{\infty}$ 
such that $\|f^{(m)}-f\|_{L^2(\T^n)} \to 0$ and $\|g^{(m)}-g\|_{L^2(\T^n)} \to 0$ as $m\to \infty$, where each $f^{(m)}$ and $g^{(m)}$ are the finite sums of elements of the form $f_1\otimes \cdots \otimes f_n$ and $g_1\otimes \cdots \otimes g_n$, respectively, and $f_{j}\in H^2(\D)^{\alpha_j}$ and $g_j\in H^2(\D)^{\beta_j}$ for all $1\leq j\leq n$. 
Since $f^{(m)}\to f$ in $L^2(\T^n)$, there exists a subsequence $\left(f^{(m_t)}\right)_{t=1}^{\infty}$ of $\left(f^{(m)}\right)_{m=1}^{\infty}$ which converges to $f$ pointwise a.e. on $\T^n$. Similarly, there exists a subsequence $\left(g^{(m_t)}\right)_{t=1}^{\infty}$ of $\left(g^{(m)}\right)_{m=1}^{\infty}$ such that $\left(g^{(m_t)}\right)_{t=1}^{\infty}$ converges to $g$ pointwise a.e. on $\T^n$. Thus
\[
f^{(m_t)}\overline{g^{(m_t)}} \to f\overline{g} \text{ pointwise a.e.\ } (e^{i\theta_1},\dots, e^{i\theta_n})\in \T^n.
\]
Suppose $E_0$ is a measurable null subset of $\T^n$ such that 
\[
f^{(m_t)}(e^{i\theta_1},\dots, e^{i\theta_n})\overline{g^{(m_t)}}(e^{i\theta_1},\dots, e^{i\theta_n}) \to f(e^{i\theta_1},\dots, e^{i\theta_n})\overline{g}(e^{i\theta_1},\dots, e^{i\theta_n}) 
\]
as $t\to \infty$ for all $(e^{i\theta_1},\dots, e^{i\theta_n})\in \T^n\setminus E_0$.
We have also shown that for each $m \geq 1$, 
\[
\left\langle P_{H^2(\D^n)}\left(f^{(m)}k_{\bxi}\right), P_{H^2(\D^n)}\left(g^{(m)}k_{\bxi}\right) \right\rangle
\to f^{(m)}\overline{g^{(m)}}\; \text{ radially. }
\]
That means there exists a measurable null subset $E_m$ of $\T^n$ such that
\[
\left\langle P_{H^2(\D^n)}\left(f^{(m)}k_{(re^{i\theta_1},\dots, re^{i\theta_n})}\right), P_{H^2(\D^n)}\left(g^{(m)}k_{(re^{i\theta_1},\dots, re^{i\theta_n})}\right) \right\rangle \to f^{(m)}(e^{i\theta_1},\dots, e^{i\theta_n})\overline{g^{(m)}}(e^{i\theta_1},\dots, e^{i\theta_n})
\]
as $r\to 1^{-}$, and $(e^{i\theta_1},\dots, e^{i\theta_n})\in \T^n\setminus E_m$.
Let $E=\bigcup_{m=0}^{\infty}E_{m}$. Then $E$ is clearly a measurable null subset of 
$\T^n$. Also $f^{(m)}k_{\bxi} \to fk_{\bxi}$ and $g^{(m)}k_{\bxi} \to gk_{\bxi}$ in $L^2(\T^n)$ because $k_{\bxi}\in L^{\infty}(\T^n)$. Hence
\[
P_{H^2(\D^n)}\left(f^{(m)}k_{\bxi}\right) \to P_{H^2(\D^n)}\left(fk_{\bxi}\right)
\; \text{ and }\; 
P_{H^2(\D^n)}\left(g^{(m)}k_{\bxi}\right) \to P_{H^2(\D^n)}\left(gk_{\bxi}\right)
\text{ in } H^2(\D^n).
\]
as  $m\to \infty$. Let $\varepsilon > 0$ be any number and fix $(e^{i\theta_1},\dots, e^{i\theta_n})\in \T^n\setminus E$. Then there exists $N_1\in \N$ such that for $m\geq N_1$ 

$|\left\langle P_{H^2(\D^n)}\left(fk_{(re^{i\theta_1},\dots, re^{i\theta_n})}\right), 
		P_{H^2(\D^n)}\left(gk_{(re^{i\theta_1},\dots, re^{i\theta_n})}\right) \right\rangle - $ 
\vspace{-5pt}
\[
\hspace{5cm}\left\langle P_{H^2(\D^n)}\left(f^{(m)}k_{(re^{i\theta_1},\dots, re^{i\theta_n})}\right), P_{H^2(\D^n)}\left(g^{(m)}k_{(re^{i\theta_1},\dots, re^{i\theta_n})}\right) \right\rangle|<\frac{\varepsilon}{3}.
\]
Now choose $M>N_1$ such that
\[
\left| f^{(m_M)}(e^{i\theta_1},\dots, e^{i\theta_n})\overline{g^{(m_M)}}(e^{i\theta_1},\dots, e^{i\theta_n}) - 
f(e^{i\theta_1},\dots, e^{i\theta_n})\overline{g}(e^{i\theta_1},\dots, e^{i\theta_n})\right| <\frac{\varepsilon}{3}.
\]  
Also there exists $s\in (0, 1)$ such that

$| \left\langle P_{H^2(\D^n)}\left(f^{(m_M)}k_{(re^{i\theta_1},\dots, re^{i\theta_n})}\right), P_{H^2(\D^n)}\left(g^{(m_M)}k_{(re^{i\theta_1},\dots, re^{i\theta_n})}\right) \right\rangle - $
\vspace{-5pt}
\[
\hspace{5cm} f^{(m_M)}(e^{i\theta_1},\dots, e^{i\theta_n})\overline{g^{(m_M)}}(e^{i\theta_1},\dots, e^{i\theta_n})| < \frac{\varepsilon}{3}
\]
for all $r\in (s, 1)$. Note that
		
$\left|\left\langle P_{H^2(\D^n)}\left(fk_{(re^{i\theta_1},\dots, re^{i\theta_n})}\right), 
P_{H^2(\D^n)}\left(gk_{(re^{i\theta_1},\dots, re^{i\theta_n})}\right) \right\rangle - 
f(e^{i\theta_1},\dots, e^{i\theta_n})\overline{g}(e^{i\theta_1},\dots, e^{i\theta_n})\right|$
\begin{align*}
&\leq |\left\langle P_{H^2(\D^n)}\left(fk_{(re^{i\theta_1},\dots, re^{i\theta_n})}\right), 
P_{H^2(\D^n)}\left(gk_{(re^{i\theta_1},\dots, re^{i\theta_n})}\right) \right\rangle \\
& \hspace{1cm} -\left\langle P_{H^2(\D^n)}\left(f^{(m_M)}k_{(re^{i\theta_1},\dots, re^{i\theta_n})}\right), P_{H^2(\D^n)}\left(g^{(m_M)}k_{(re^{i\theta_1},\dots, re^{i\theta_n})}\right) \right\rangle|\\
&+ | \left\langle P_{H^2(\D^n)}\left(f^{(m_M)}k_{(re^{i\theta_1},\dots, re^{i\theta_n})}\right), P_{H^2(\D^n)}\left(g^{(m_M)}k_{(re^{i\theta_1},\dots, re^{i\theta_n})}\right) \right\rangle \\
& \hspace{1cm} -f^{(m_M)}(e^{i\theta_1},\dots, e^{i\theta_n})\overline{g^{(m_M)}}(e^{i\theta_1},\dots, e^{i\theta_n})| \\
&+ \left| f^{(m_M)}(e^{i\theta_1},\dots, e^{i\theta_n})\overline{g^{(m_M)}}(e^{i\theta_1},\dots, e^{i\theta_n}) - f(e^{i\theta_1},\dots, e^{i\theta_n})\overline{g}(e^{i\theta_1},\dots, e^{i\theta_n})\right|.
\end{align*}
Hence
\begin{align*}
|\left\langle P_{H^2(\D^n)}\left(fk_{(re^{i\theta_1},\dots, re^{i\theta_n})}\right),
P_{H^2(\D^n)}\left(gk_{(re^{i\theta_1},\dots, re^{i\theta_n})}\right) \right\rangle & \\ 
&\hspace{-3cm} -f(e^{i\theta_1},\dots, e^{i\theta_n})\overline{g}(e^{i\theta_1},\dots, e^{i\theta_n})| < \frac{\varepsilon}{3} + \frac{\varepsilon}{3} + \frac{\varepsilon}{3}
= \varepsilon
\end{align*}
whenever $r\in (s, 1)$. 

This completes the proof.
\end{proof}

We have an immediate consequence of the above result as follows.

\begin{cor}\label{cor:my corollary}
If $f,g\in L^2(\T^n)$, then 
\[
\left\langle P_{H^2(\D^n)}(fk_{\bxi}), P_{H^2(\D^n)}(gk_{\bxi}) \right\rangle
\to f\overline{g}\; \text{ radially.}
\]
\end{cor}

\begin{proof}
Note that
\[
L^2(\T^n)=\bigoplus_{\alpha\in\{-1, 1\}^n}\negthickspace\negthickspace H^2(\D^n)^{\alpha}.
\]
Let $f, g \in L^2(\T^n)$ such that
\[
f=\bigoplus_{\alpha\in\{-1,1\}^n}\negthickspace\negthickspace f_{\alpha} 
\quad \text{ and } \quad
g=\bigoplus_{\beta\in\{-1,1\}^n}\negthickspace\negthickspace g_{\beta},
\]
where $f_{\alpha}\in H^2(\D^n)^{\alpha}$ and $g_{\beta}\in H^2(\D^n)^{\beta}$ for all 
$\alpha,\beta\in\{-1,1\}^n$. Therefore,
\begin{align*}
\left\langle P_{H^2(\D^n)}(fk_{\bxi}), P_{H^2(\D^n)}(gk_{\bxi}) \right\rangle
&= \sum_{\alpha\in\{-1,1\}^n}\sum_{\beta\in\{-1,1\}^n}\left\langle P_{H^2(\D^n)}    
(f_{\alpha}k_{\bxi}), P_{H^2(\D^n)}(g_{\beta}k_{\bxi}) \right\rangle\\
&\to \sum_{\alpha\in\{-1,1\}^n}\sum_{\beta\in\{-1,1\}^n}f_{\alpha}\overline{g_{\beta}}=f\overline{g} \quad \text{ radially}, 
\end{align*}
which follows from Theorem \ref{thm:my theorem}.
\end{proof}

\section{Symbol map for $H^2(\D^n)$}

In this section we shall discuss about the \emph{symbol map} on the Toeplitz algebra on 
$H^2(\D^n)$. Before that, we need some results in the sequel.

\begin{thm}\label{thm:Hankel radial limit 0}
For $\phi\in L^{\infty}(\T^n)$, $\|H_{\phi}k_{\bxi}\| \to 0$ radially. Moreover, if 
$\phi, \psi\in L^{\infty}(\T^n)$, then $\widetilde{[T_{\phi}, T_{\psi})} \to 0$ radially. 
\end{thm}

\begin{proof}
Suppose $\phi\in L^{\infty}(\T^n)$. Then by definition
\begin{align*}
H_{\phi}k_{\bxi} &= P_{H^2(\D^n)^\perp}(\phi k_{\bxi})\\
&= P_{H^2(\D^n)^\perp}\left((P_{H^2(\D^n)}\phi+ P_{H^2(\D^n)^\perp}\phi)k_{\bxi}\right)\\
&= P_{H^2(\D^n)^\perp}\left((P_{H^2(\D^n)}\phi)k_{\bxi}\right) +P_{H^2(\D^n)^\perp}\left((P_{H^2(\D^n)^\perp}\phi)k_{\bxi}\right).
\end{align*}
Since $k_{\bxi}\in H^{\infty}(\D^n)$, clearly $(P_{H^2(\D^n)}\phi)k_{\bxi}\in H^2(\D^n)$.
Therefore, 
\[
H_{\phi}k_{\bxi}=P_{H^2(\D^n)^\perp}\left((P_{H^2(\D^n)^\perp}\phi)k_{\bxi}\right).
\]
Set $y=P_{H^2(\D^n)^\perp}\phi$. Then
\[
H_{\phi}k_{\bxi}=P_{H^2(\D^n)^\perp}(yk_{\bxi}).
\]
Again
\[ 
yk_{\bxi}=P_{H^2(\D^n)}(yk_{\bxi})+P_{H^2(\D^n)^\perp}(yk_{\bxi}).
\] 
Thus
$\|yk_{\bxi}\|^2=\|P_{H^2(\D^n)}(yk_{\bxi})\|^2+\|P_{H^2(\D^n)^\perp}(yk_{\bxi})\|^2$,
and hence
\[
\|H_{\phi}k_{\bxi}\|^2=\|yk_{\bxi}\|^2-\|P_{H^2(\D^n)}(yk_{\bxi})\|^2.
\]
Now Lemma \ref{lemma:radial limit lemma} yields	
\[
\|yk_{\bxi}\|^2=\langle yk_{\bxi}, yk_{\bxi} \rangle = \widetilde{|y|^2}(\bxi). 
\]
Therefore,
\begin{equation} \label{eqn:Hankel radial limit 0}
\|H_{\phi}k_{\bxi}\|^2=\widetilde{|y|^2}(\bxi)-\|P_{H^2(\D^n)}(yk_{\bxi})\|^2.
\end{equation}
Now Corollary \ref{cor:my corollary} yields
\[
\|P_{H^2(\D^n)}(yk_{\bxi})\|^2 \to |y|^2 \quad \text{radially}.
\] 
Hence, from equation \eqref{eqn:Hankel radial limit 0}, we conclude that
\[
\|H_{\phi}k_{\bxi}\| \to 0 \quad \text{ radially}.
\]

Note that
\begin{align*}
\left|\widetilde{[T_{\phi}, T_{\psi})}(\bxi)\right| 
= \left|\langle H_{\overline{\phi}}^{*}H_{\psi}k_{\bxi}, k_{\bxi} \rangle\right| 
= \left|\langle H_{\psi}k_{\bxi}, H_{\overline{\phi}}k_{\bxi} \rangle\right| 
\leq \|H_{\psi}k_{\bxi}\|\,\|H_{\overline{\phi}}k_{\bxi}\| \to 0 \quad \text{ radially.}
\end{align*}
This finishes the proof of the second part.
\end{proof}

\begin{thm} \label{thm:zero radial limit of operators in semicommutator ideal}
Let $S\in \Sscr$. Then $\widetilde{S} \to 0$ radially.
\end{thm}

\begin{proof}
Note that the set of all finite sums of operators of the form $A[T_{\phi}, T_{\psi})B$, where $A, B\in \Tscr(L^{\infty}(\T^n))$, forms a dense subset of $\Sscr$. Now the set of all finite sums of operators of the form $T_{\phi_1}T_{\phi_2}\cdots T_{\phi_n}$ is dense in $\Tscr(L^{\infty}(\T^n))$ and hence the span of the operators of the form
\[
T_{\phi_1}T_{\phi_2}\cdots T_{\phi_n}[T_{\phi}, T_{\psi})T_{\psi_1}T_{\psi_2}\cdots T_{\psi_m}
\]
form a dense subset of $\Sscr$. Therefore, it suffices to prove for $S=T_{\phi_1}T_{\phi_2}\cdots T_{\phi_n}[T_{\phi}, T_{\psi})T_{\psi_1}T_{\psi_2}\cdots T_{\psi_m}$, where $\phi, \phi_i, \psi, \psi_i\in L^{\infty}(\T^n)$. Again for any symbol 
$u, v, w \in L^{\infty}(\T^n)$,
\begin{align*}
T_w[T_u, T_v) &= T_wT_{uv}-T_wT_uT_v \notag \\
&= (T_w T_{uv}-T_{wuv}) + (T_{wuv}-T_{wu}T_v) + (T_{wu}T_v-T_wT_uT_v) \notag \\
&= -[T_w, T_{uv}) + [T_{wu}, T_v) + [T_w, T_u)T_v. 
\end{align*}
Thus, it is enough to assume that $S\in \Sscr$ is of the form 
\[
S=[T_{\phi}, T_{\psi})A = H_{\overline{\phi}}^{*}H_{\psi}A \qquad 
(A\in \Tscr(L^{\infty}(\T^n)),  \phi, \psi\in L^{\infty}(\T^n)).
\]
Therefore,
\begin{align*}
\left| \widetilde{S}(\bxi) \right| 
&= \left|\langle H_{\overline{\phi}}^{*}H_{\psi}Ak_{\bxi}, k_{\bxi} \rangle\right| \\
&= \left|\langle H_{\psi}Ak_{\bxi}, H_{\overline{\phi}}k_{\bxi} \rangle\right| \\
&= \| H_{\psi}Ak_{\bxi}\|\, \|H_{\overline{\phi}}k_{\bxi}\| \\
&\leq \|H_{\psi}A\|\, \|H_{\overline{\phi}}k_{\bxi}\| \to 0 \quad \text{ radially, }
\end{align*}
which follows from Theorem \ref{thm:Hankel radial limit 0}. 
\end{proof}

We are now in a position to state the main result of this section.
	
\begin{thm} \label{thm:douglas theorem for polydisk}
Let $T\in \Tscr(L^{\infty}(\T^n))$. Then $\widetilde{T}\to \phi$ radially for some $\phi\in L^{\infty}(\T^n)$. Moreover, the map 
\[
\Sigma : \Tscr(L^{\infty}(\T^n)) \to L^{\infty}(\T^n), \qquad T \mapsto \phi
\]
is a $C^*$-algebra homomorphism with ker$(\Sigma)=\Sscr$ and $\Sigma(T_{\phi})=\phi$ for any Toeplitz operator $T_{\phi}$.
\end{thm}

\begin{proof}
We shall prove by induction that 
\[
\mbox{for any $m \geq 2$ if~} \phi_1,  \dots, \phi_m \in L^{\infty}(\T^n), ~\mbox{then~}
T_{\phi_1} \cdots T_{\phi_m}-T_{\phi_1 \cdots \phi_m}\in \Sscr.
\]

Recall that for any two symbols $\phi_1, \phi_2 \in L^{\infty}(\T^n)$, 
$T_{\phi_1}T_{\phi_2}-T_{\phi_1\phi_2}=-[T_{\phi_1}, T_{\phi_2})\in \Sscr$.

Suppose now $T_{\phi_1} \cdots T_{\phi_m}-T_{\phi_1 \cdots \phi_m}\in \Sscr$ for some $m >2$ and $\phi_1, \dots, \phi_m \in L^{\infty}(\T^n).$

Note that 
\begin{align*}
T_{\phi_1} \cdots T_{\phi_m}T_{\phi_{m+1}}-T_{\phi_1 \cdots \phi_m\phi_{m+1}} &= (T_{\phi_1} \cdots T_{\phi_m}T_{\phi_{m+1}} - T_{\phi_1 \cdots \phi_m}T_{\phi_{m+1}}) \\
& +(T_{\phi_1 \cdots \phi_m}T_{\phi_{m+1}}-T_{(\phi_1 \cdots \phi_m)\phi_{m+1}}) \\
&= (T_{\phi_1} \cdots T_{\phi_m}-T_{\phi_1 \cdots \phi_m})T_{\phi_{m+1}} \\
&-[T_{\phi_1 \cdots \phi_m}, T_{\phi_{m+1}}). 
\end{align*}
Since $\Sscr$ is a closed ideal in $\Tscr(L^{\infty}(\T^n))$, we have
\[
T_{\phi_1} \cdots T_{\phi_m}T_{\phi_{m+1}}-T_{\phi_1 \cdots \phi_m\phi_{m+1}}\in \Sscr.
\]
Now the collection of all finite sums of operators of the form $T_{\phi_1}\cdots T_{\phi_m}$ is dense in $\Tscr(L^{\infty}(\T^n))$. Thus operators of the form
\[
T=T_{\phi}+S \qquad (\phi\in L^{\infty}(\T^n),~ S\in \Sscr)
\]
is a dense subset of $\Tscr(L^{\infty}(\T^n))$. Consider
\[
\Tscr_0(L^{\infty}(\T^n))=\{\,T_{\phi}+S : \phi\in L^{\infty}(\T^n),\,S\in \Sscr\,\}.
\]
Suppose $T=T_{\phi}+S$ for some $\phi\in L^{\infty}(\T^n)$ and $S\in \Sscr$. 
Then, Theorem \ref{thm:radial limit of Toeplitz operaor} and Theorem \ref{thm:zero radial limit of operators in semicommutator ideal} yield $\widetilde{T} \to \phi$ radially. Define a map $\Sigma_0 : \Tscr_0(L^{\infty}(\T^n)) \to L^{\infty}(\T^n)$ as
\begin{align}\label{eq:definition of sigma_0}
\Sigma_0(T) = \phi \quad (T=T_{\phi}+S \in \Tscr_0(L^{\infty}(\T^n))).
\end{align}
We shall show that ${\Sigma_0}$ is well-defined. Suppose $T_{\phi_1}+S_1=T_{\phi_2}+S_2$. Then 
\begin{align*}
T_{\phi_1}-T_{\phi_2}=S_2-S_1 \implies T_{\phi_1-\phi_2}=S_2-S_1 .
\end{align*}
Now Theorem \ref{thm:zero radial limit of operators in semicommutator ideal} gives  
$\widetilde{T}_{\phi_1-\phi_2}=\widetilde{S_2-S_1} \to 0\text{ radially}$. Again 
Theorem \ref{thm:radial limit of Toeplitz operaor} infers $ \phi_1-\phi_2=0 \text{ a.e.}$ and hence $\phi_1=\phi_2 \text{ a.e.}$.

Suppose $T_{\phi_1}+S_1, T_{\phi_2}+S_2 \in \Tscr_0(L^{\infty}(\T^n)) $ and $\alpha \in \C$. Then
\begin{align*}
\Sigma_0\left( \alpha(T_{\phi_1}+S_1)+(T_{\phi_2}+S_2) \right) 
&= \Sigma_0(T_{\alpha\phi_1+\phi_2}+\alpha S_1+S_2) \\
&= \alpha\phi_1+\phi_2 \tag{By \eqref{eq:definition of sigma_0}} \\
&= \alpha\Sigma_0(T_{\phi_1}+S_1)+\Sigma_0(T_{\phi_2}+S_2).
\end{align*}
Again if $T \in \B(H^2(\D^n))$, then
\[
\left| \widetilde{T}(\bxi) \right|=\left|\langle Tk_{\bxi}, k_{\bxi} \rangle \right|\leq\|T\| \qquad  ( \bxi\in \D^n).
\]
Suppose $T=T_{\phi}+S$ for $\phi\in L^{\infty}(\T^n)$ and $S\in \Sscr$. Then
\[
\left|\widetilde{T}(re^{i\theta_1},\dots, re^{i\theta_n})\right|\leq\|T\| \quad 
\mbox{for all~} 0\leq r <1,\, (e^{i\theta_1}, \dots, e^{i\theta_n})\in \T^n.
\]
Letting $r\to 1^-$, we obtain 
$\left|\phi(e^{i\theta_1},\dots, e^{i\theta_n})\right|\leq \|T\|$ a.e.\ $(e^{i\theta_1},\dots, e^{i\theta_n})\in \T^n$. Hence
\[
\|\Sigma_0(T)\|_{L^{\infty}(\T^n)}=\|\phi\|_{L^{\infty}(\T^n)}\leq\|T\|.
\]
Therefore, $\Sigma_0$ is a linear and bounded map from $\Tscr_0(L^{\infty}(\T^n))$ to $L^{\infty}(\T^n)$.

We now extend $\Sigma_0$ to $ \overline{\Tscr_0(L^{\infty}(\T^n))}= \Tscr(L^{\infty}(\T^n))$, denoting its extension by $\Sigma$. Thus $\Sigma : \Tscr(L^{\infty}(\T^n)) \to L^{\infty}(\T^n)$ is a linear and bounded map. Suppose $T\in \Tscr(L^{\infty}(\T^n))$ and set $\phi=\Sigma(T)$. Then there exists a sequence $\{T_{m}\}_{m=1}^{\infty}$, where each $T_m=T_{\phi_m}+S_{m}\in \Tscr_0(L^{\infty}(\T^n))$, such that $\|T_m-T\|\to 0$ as $m\to \infty$. Therefore, 
\[
\|T_{\phi_m}-T_{\phi}\|=\|T_{\phi_m-\phi}\|=\|\phi_m-\phi\|_{L^{\infty}(\T^n)}=\|\Sigma(T_m)-\Sigma(T)\|_{L^{\infty}(\T^n)}\to 0 ~~ \mbox{~as~} m \to \infty.
\]
Thus
\[
S_m =T_m-T_{\phi_m}\to T-T_{\phi}  \text{ as } m\to \infty. 
\]    
Let $S=T-T_{\phi}$. Then $S\in \Sscr$ as $\{ S_m \}$ in $\Sscr$ and $\Sscr$ is closed, 
and hence if $T\in \Tscr(L^{\infty}(\T^n))$, then $T=T_{\phi}+S$, where $\phi\in L^{\infty}(\T^n)$ and $S\in \Sscr$. Thus
\[
\Tscr(L^{\infty}(\T^n))=\{\,T_{\phi}+S : \phi\in L^{\infty}(\T^n),\,S\in \Sscr\,\}.
\]
Therefore, Theorem \ref{thm:radial limit of Toeplitz operaor} and Theorem \ref{thm:zero radial limit of operators in semicommutator ideal} give $\widetilde{T} \to \phi$ radially.

To show $\Sigma$ is a $C^*$-algebra homomorphism, let $T=T_{\phi}+S$, where $\phi\in L^{\infty}(\T^n)$ and $S\in \Sscr$. 		
Note that 
\begin{align*}
\widetilde{(T_{\phi}+S)^*} &= \widetilde{T_{\overline{\phi}}+S^*} 
= \widetilde{T_{\overline{\phi}}}+\widetilde{S^*}
= \widetilde{T_{\overline{\phi}}}+\overline{\widetilde{S}} \to \overline{\phi} \quad radially.
\end{align*}
Therefore, $\Sigma(T^*)=\overline{\Sigma(T)}$. Suppose $T_1=T_{\phi_1}+S_1$ and $T_2=T_{\phi_2}+S_2$, where $\phi_1, \phi_2 \in L^{\infty}(\T^n)$ and $S_1, S_2 \in \Sscr$. Then
\begin{align*}
T_1T_2 
&= T_{\phi_1}T_{\phi_2}+T_{\phi_1}S_2+S_1T_{\phi_2}+S_1S_2 \\
&= T_{\phi_1\phi_2}+(T_{\phi_1}T_{\phi_2}-T_{\phi_1\phi_2})
+(T_{\phi_1}S_2+S_1T_{\phi_2}+S_1S_2).
\end{align*}
Since $\Sscr$ is an ideal in $\Tscr(L^{\infty}(\T^n))$, clearly $(T_{\phi_1}T_{\phi_2}-T_{\phi_1\phi_2}) +(T_{\phi_1}S_2+S_1T_{\phi_2}+S_1S_2) \in \Sscr$. Hence 
\[
\widetilde{T_1T_2} \to \phi_1\phi_2 ~~\mbox{radially}.
\]
It follows that $\Sigma(T_1T_2) = \Sigma(T_1) \Sigma(T_2)$ for any $T_1, T_2 \in \Tscr(L^{\infty}(\T^n))$.

Finally, it is clear from the definition that ker($\Sigma$)$=\Sscr$, the semicommutator ideal $\Sscr$ in $\Tscr(L^{\infty}(\T^n))$, and $\Sigma(T_{\phi})=\phi$ for all $\phi\in L^{\infty}(\T^n)$.	
This finishes the proof.	
\end{proof}

The next result readily follows from the above theorem.	
	
\begin{cor}\label{cor: (T^*T-TT^*) to 0}
For $T\in \Tscr(L^{\infty}(\T^n))$, $\widetilde{[T^*, T]} \to 0$ radially.
\end{cor}

We shall now show that the symbol map can not be extended to the whole of 
$\B(H^2(\D^n))$ in the following example (see \cite{{Engls-SYMBOL MAP VIA BEREZIN 
TRAN. ON H2}} for $n=1$ case).

\begin{ex}	
Consider the operator $T$ on $H^2(\D)$, defined by
\[
Tz^{k}=z^{2k+1} \qquad (k\in \Z_{+}),
\]
where $\{z^k \}_{k\in \Z_{+}}$ denotes the standard orthonormal basis for $H^2(\D)$. 
Now
\[
\langle Tz^k, Tz^l \rangle=\langle z^{2k+1}, z^{2l+1}\rangle=\delta_{k,l}=\langle z^k, z^l\rangle, \qquad (k,l\in\Z_+).
\]
Therefore, $T$ is an isometry on $H^2(\D)$. Thus, $T^*T=I_{H^2(\D)}$ and $TT^*=P_{\mathcal{R}(T)}$, the orthogonal projection onto the range $\mathcal{R}(T)=\bigvee_{k=0}^{\infty}\{z^{2k+1}\}$. Hence
\[
(T^*T-TT^*)z^k=(I-P_{\mathcal{R}(T)})z^k=
\begin{cases}
0 \quad & \text{if $k$ is odd;}\\
z^k \quad & \text{if $k$ is even.}
\end{cases}
\]
That means $[T^*, T]$ is an orthogonal projection onto $\clm=\bigvee\{z^{2k} : k\in \Z_{+}\}$. Therefore, for $\xi \in \D$
\begin{align*}
\widetilde{[T^*, T]}(\xi)
=\|P_{\mathcal{M}}k_{\xi}\|^2
&=(1-|\xi|^2)\|P_{\mathcal{M}}K_{\xi}\|^2\\
&=(1-|\xi|^2)\left\|P_{\mathcal{M}}\left(\sum_{k=0}^{\infty}\overline{\xi}^kz^k\right)\right\|^2\\
&=(1-|\xi|^2)\left\|\sum_{k=0}^{\infty}\overline{\xi}^{2k}z^{2k}\right\|^2\\
&=(1-|\xi|^2)\sum_{k=0}^{\infty}|\xi|^{4k}\\
&=\frac{(1-|\xi|^2)}{(1-|\xi|^4)}=\frac{1}{1+|\xi|^2} \to \frac{1}{2} \text{ radially.}
\end{align*}
Now we consider the operator $S$ on $H^2(\D^n)$ as
\[
S =\overbrace{T\otimes \cdots \otimes T}^{n\text{-times}}.
\]
Then 
\begin{align*}
S^*S-SS^* &= (T\otimes \cdots \otimes T)^*(T\otimes \cdots \otimes T)-(T\otimes \cdots \otimes T)(T\otimes  \cdots \otimes T)^*\\
&= (T^*T\otimes T^*T\otimes \cdots \otimes T^*T)-(TT^*\otimes TT^*\otimes \cdots \otimes TT^*)\\
&= I_{H^2(\D^n)}-(TT^*\otimes TT^*\otimes \cdots \otimes TT^*).
\end{align*}
Thus for $\bm{k}=(k_1,\dots,k_n)\in \Z_{+}^n$,
\[
[S^*, S]\z^{\K} =
\begin{cases}
\z^{\K} \quad & \text{if $k_j$ is even for some $j$;}\\
0       \quad & \text{if $k_j$ is odd for all $j$.} 
\end{cases}
\]
Hence $[S^*, S]$ is an orthogonal projection onto
\[
\cln=\bigvee \{ \z^{\K}  : \bm{k}=(k_1,\dots, k_n)\in \Z_{+}, k_j \text{ is even for some } 1 \leq j \leq n \},
\]
that is, $[S^*, S]= P_{\cln}$.
		
\NI
Now for $\bxi=(\xi_1,\dots,\xi_n)\in \D^n$,
\begin{align*}
\widetilde{[S^*, S]}(\bxi) =\|P_{\mathcal{N}}k_{\bxi}\|^2
&=\left( \prod_{j=1}^n(1-|\xi_j|^2)\right)\left\|P_{\cln}\left( \sum_{\K\in\Z_+^n}\overline{\bxi}^{\K}\z^{\K} \right)\right\|^2\\
&=\left( \prod_{j=1}^n(1-|\xi_j|^2)\right)\left\| \sum_{\underset{\text{some $k_j$ even}}{\K\in\Z_+^n}}\overline{\bxi}^{\K}\z^{\K}\right\|^2\\
&=\left( \prod_{j=1}^n(1-|\xi_j|^2)\right)\left( \sum_{\text{some $k_j$ even}}|\xi_1|^{2k_1}\dots|\xi_n|^{2k_n} \right)\\
&\geq\left( \prod_{j=1}^n(1-|\xi_j|^2)\right)\left( \sum_{\text{all $k_j$ even}}|\xi_1|^{2k_1}\dots|\xi_n|^{2k_n} \right)\\
&=\lt( \prod_{j=1}^n(1-|\xi_j|^2)\rt) \lt( \sum_{k_1 \text{ even}} |\xi_1|^{2k_1} \rt)\dots \lt( \sum_{k_n \text{ even}} |\xi_n|^{2k_n} \rt)\\
&=\lt( \prod_{j=1}^n(1-|\xi_j|^2)\rt) \lt( \sum_{m=0}^{\infty} |\xi_1|^{4m} \rt)\dots
\lt( \sum_{m=0}^{\infty}|\xi_n|^{4m} \rt)\\
&=\frac{\prod_{j=1}^n(1-|\xi_j|^2)}{\prod_{j=1}^n(1-|\xi_j|^4)}
={\prod_{j=1}^n(1+|\xi_j|^2)^{-1}} \geq \frac{1}{2^n} > 0.
\end{align*}
Hence, by corollary \ref{cor: (T^*T-TT^*) to 0} we conclude that the symbol calculus can not be extended to any $C^*$-algebra that contains $S$.
\end{ex}

\section{Extension of Toeplitz Algebra}

In this section, we present a family of $C^*$-algebras containing the Toeplitz algebra 
$\Tscr(L^{\infty}(\T^n))$ and the analogs of Theorem \ref{thm:douglas theorem for polydisk} still hold.

Let $\psi\in L^{\infty}(\T^n)$ and $S\in \Sscr$. Consider
\begin{equation}\label{eq:Iscr algebra}
\Bscr(\psi, S) = \{A\in \B(H^2(\D^n)) : \|A(T_{\psi}+S)k_{\bxi}\| \text{ and } \|A^*(T_{\psi}+S)k_{\bxi}\| \to 0 \text{ radially}\}.    
\end{equation}
In particular, if $\psi = 1$ a.e.\ on $\T^n$ and $S=\bm{0}$, where $\bm{0}$ denotes the zero operator on $H^2(\D^n)$, then we simply write $\Bscr=\Bscr(1, \bm{0})$.

\begin{lem}\label{B is c* algebra}
$\Bscr(\psi,\bm{0})$ is a $C^*$-algebra for all $\psi\in L^{\infty}(\T^n)$.
\end{lem}

\begin{proof}
Let $A_1, A_2\in \Bscr(\psi,\bm{0})$ and $c\in \C$. Then $cA_1+A_2, A_1 A_2\in \Bscr(\psi,\bm{0})$
as	
\begin{align*}	
\|(cA_1+A_2)T_{\psi}k_{\bxi}\| &\leq |c|\|A_1T_{\psi}k_{\bxi}\|+\|A_2T_{\psi}k_{\bxi}\|,\\
\|(cA_1+A_2)^*T_{\psi}k_{\bxi}\|& \leq|c|\|A_1^*T_{\psi}k_{\bxi}\|+\|A_2^*T_{\psi}k_{\bxi}\|,
\end{align*}
\[
\|A_1A_2T_{\psi}k_{\bxi}\|\leq\|A_1\|\|A_2T_{\psi}k_{\bxi}\| \quad \text{and} \quad 
\|(A_1A_2)^*T_{\psi}k_{\bxi}\|\leq\|A_2^*\|\|A_1^*T_{\psi}k_{\bxi}\|.
\]
From the definition, it follows that if $A \in \Bscr(\psi,\bm{0})$, then  $A^* \in \Bscr(\psi,\bm{0})$.

To see that $\Bscr(\psi,\bm{0})$ is closed, let $(A_m)$ be a sequence in $\Bscr(\psi,\bm{0})$ such that $\|A-A_m\|\to 0$ as $m\to \infty$ for some $A\in\B(H^2(\D^n))$. 
Now for any $\varepsilon>0$ choose $N\in\N$ such that $\|A-A_N\|<\frac{\varepsilon}{2(1+\|T_{\psi} \|)}$. Therefore,
\begin{align}
\|AT_{\psi}k_{\bxi}\| 
&\leq \|(A-A_N)T_{\psi}k_{\bxi}\|+\|A_NT_{\psi}k_{\bxi}\| \notag \\
&\leq \|A-A_N\| \|T_{\psi} \| + \|A_NT_{\psi}k_{\bxi}\| \notag \\
&< \varepsilon/2 + \|A_NT_{\psi}k_{\bxi}\|. \label{eq:Iscr is closed}
\end{align}
Since $\|A_mT_{\psi}k_{\bxi}\|\to 0$ radially for each $m\in \N$, there exists a measure zero subset $E_m$ of $\T^n$ such that 
\[
\|A_mT_{\psi}k_{(re^{i\theta_1},\dots,re^{i\theta_n})}\|\to 0 \quad\text{as $r \to 1^{-}$},
\]
for all $(e^{i\theta_1},\dots,e^{i\theta_n})\in\T^n\setminus E_m$. Consider $E=\cup_{m=1}^{\infty}E_m$ of $\T^n$, and hence $E$ is a measure zero subset. Thus for a fixed 
$(e^{i\theta_1},\dots,e^{i\theta_n})\in \T^n\setminus E$, we choose $0<s<1$ such that 
$\|A_NT_{\psi}k_{(re^{i\theta_1},\dots,re^{i\theta_n})}\|<\varepsilon/2$ whenever $s<r<1$. Hence the equation \eqref{eq:Iscr is closed} implies that
\[
s<r<1 \implies \|AT_{\psi}k_{(re^{i\theta_1},\dots,re^{i\theta_n})}\|<\varepsilon.
\]
Therefore, $\|AT_{\psi}k_{\bxi}\|\to 0$ radially. Now $\|A_m^*-A^*\|=\|A_m-A\|\to 0$ as $m\to \infty$ and $\|A_m^*T_{\psi}k_{\bxi}\|\to 0$ radially for all $m\in\N$. It follows that $\|A^*T_{\psi}k_{\bxi}\|\to 0$ radially. Thus $A\in\Bscr(\psi,\bm{0})$ and hence 
$\Bscr(\psi,\bm{0})$ is closed.
\end{proof}

\begin{lem}\label{lem: B contains semicommutator}
$\Cscr\subseteq\Sscr\subseteq\Bscr$.
\end{lem}

\begin{proof}
Firstly, it is easy to see that $\Cscr\subseteq\Sscr$. Let $\phi,\psi\in L^{\infty}(\T^n)$. Then
\begin{align*}
\|[T_{\phi}, T_{\psi})k_{\bxi}\| &= \|H_{\overline{\phi}}^*H_{\psi}k_{\bxi}\|\leq\|H_{\overline{\phi}}^*\|\|H_{\psi}k_{\bxi}\|, \\
\text{ and } \|[T_{\phi},T_{\psi})^*k_{\bxi}\| &= \|(H_{\overline{\phi}}^*H_{\psi})^*k_{\bxi}\|
\leq\|H_{\psi}^*\|\|H_{\overline{\phi}}k_{\bxi}\|.
\end{align*}
Now Theorem \ref{thm:Hankel radial limit 0} infers that $[T_{\phi},T_{\psi})\in\Bscr$,
and hence $\Cscr\subseteq\Sscr\subseteq\Bscr$.
\end{proof}

\begin{lem}
$\Bscr(\psi, S)=\Bscr(\psi, \bm{0})$ for all $\psi\in L^{\infty}(\T^n)$ and $S\in \Sscr$.
\end{lem}

\begin{proof}
It follows easily from Lemma \ref{lem: B contains semicommutator} that $\Bscr(\psi, \bm{0})\subseteq \Bscr(\psi, S)$. 

Suppose $\psi\in L^{\infty}(\T^n)$ and $S\in \Sscr$, and let $A\in \Bscr(\psi, S)$. 
Then $\|Sk_{\bxi}\|\to 0$ radially, and we also have
\begin{align*}
\|AT_{\psi}k_{\bxi}\| 
&\leq \|A(T_{\psi}+S)k_{\bxi}\| + \|A\|\|Sk_{\bxi}\|\to 0,\\
\|A^*T_{\psi}k_{\bxi}\|
&\leq \|A^*(T_{\psi}+S)k_{\bxi}\| + \|A\|\|Sk_{\bxi}\|\to 0.
\end{align*}
Therefore, $A\in \Bscr(\psi, \bm{0})$.
\end{proof}

\begin{thm} \label{thm:eigenfunctions of Toeplitz operator}
Let $\phi\in L^2(\T^n)$. Then $\|P_{H^2(\D^n)}(\phi k_{\bxi}) - \widetilde{\phi}(\bxi)k_{\bxi}\|\to 0$ radially. In particular, $\|T_{\phi}k_{\bxi}-\wt{\phi}(\bxi)k_{\bxi}\|\to 0$ radially for $\phi\in L^{\infty}(\T^n)$.
\end{thm}

\begin{proof}
First, assume that $\phi \in L^2(\T^n)$. Then we have
\begin{align*}
&\|P_{H^2(\D^n)}(\phi k_{\bxi})-\wt{\phi}(\bxi)k_{\bxi}\|^2\\
&= \|P_{H^2(\D^n)}(\phi k_{\bxi})\|^2+\|\wt{\phi}(\bxi)k_{\bxi}\|^2
-\la P_{H^2(\D^n)}(\phi k_{\bxi}), \wt{\phi}(\bxi)k_{\bxi} \ra
-\la \wt{\phi}(\bxi)k_{\bxi}, P_{H^2(\D^n)}(\phi k_{\bxi}) \ra \\
&= \|P_{H^2(\D^n)}(\phi k_{\bxi})\|^2+\|\wt{\phi}(\bxi)k_{\bxi}\|^2
-\ol{\wt{\phi}(\bxi)}\la P_{H^2(\D^n)}(\phi k_{\bxi}), k_{\bxi} \ra
-\wt{\phi}(\bxi)\la k_{\bxi}, P_{H^2(\D^n)}(\phi k_{\bxi}) \ra. 
\end{align*}
Now Corollary \ref{cor:my corollary} yields
\[
\|P_{H^2(\D^n)}(\phi k_{\bxi})-\wt{\phi}(\bxi)k_{\bxi}\| \to 0 \quad \text{ radially}.
\]

Finally, if $\phi\in L^{\infty}(\T^n)$, then $P_{H^2(\D^n)}(\phi k_{\bxi})=T_{\phi}k_{\bxi}$ and  hence the required result follows.
\end{proof}

\begin{prop}\label{prop:B is ideal w.r.to Toeplitz}
If $\phi\in L^{\infty}(\T^n)$ and $A\in\Bscr(\psi,\bm{0})$, then $T_{\phi}A, AT_{\phi}\in \Bscr(\psi,\bm{0})$.    
\end{prop}

\begin{proof}
To show $T_{\phi}A$ and $AT_{\phi}$ are in $\Bscr(\psi,\bm{0})$ for $A\in\Bscr(\psi,\bm{0})$ and $\phi\in L^{\infty}(\T^n)$, it suffices to check that $\|AT_{\phi}T_{\psi}k_{\bxi}\|\to 0$ radially. Note that
\begin{align}
\|AT_{\phi}T_{\psi}k_{\bxi}\|
&\leq \|AT_{\phi}T_{\psi}k_{\bxi}-\widetilde{\phi}(\bxi)AT_{\psi}k_{\bxi}\|+\|\wt{\phi}             
(\bxi)AT_{\psi}k_{\bxi}\|   \notag \\
&\leq \|A\|\|T_{\phi}T_{\psi}k_{\bxi}-\wt{\phi}
(\bxi)T_{\psi}k_{\bxi}\|+\|\phi\|_{\infty}\|AT_{\psi}k_{\bxi}\|. \label{eq:B is ideal w.r.to Toeplitz}
\end{align}
Also 
\[
\|T_{\phi}T_{\psi}k_{\bxi}-\widetilde{\phi}(\bxi)T_{\psi}k_{\bxi}\|^2=(T_{\psi}^*T_{\phi}^*T_{\phi}T_{\psi}\widetilde{)}(\bxi)+|\widetilde{\phi}(\bxi)|^2\|T_{\psi}k_{\bxi}\|^2-2\text{Re}\left(\overline{\widetilde{\phi}(\bxi)}(T_{\psi}^*T_{\phi}T_{\psi}\widetilde{)}(\bxi)\right). 
\]
Applying Theorems \ref{thm: radial limit theorem}, \ref{thm:douglas theorem for polydisk}, and \ref{thm:radial limit of Toeplitz operaor}, we obtain 
\[
\|T_{\phi}T_{\psi}k_{\bxi}-\widetilde{\phi}(\bxi)T_{\psi}k_{\bxi}\|^2
\to |\phi \psi|^2+|\phi \psi|^2-2|\phi \psi|^2=0 \quad \text{radially.}
\]
Hence $\|AT_{\phi}T_{\psi}k_{\bxi}\|\to 0$ radially follows from \eqref{eq:B is ideal w.r.to Toeplitz}. 
\end{proof}

An immediate consequence of Proposition \ref{prop:B is ideal w.r.to Toeplitz} is the following.

\begin{lem}
Let $\psi_1$ and $\psi_2$ be two inner functions on $\D^n$ and let $\psi=\overline{\psi_1}\psi_2$. Then $\Bscr(\psi,\bm{0})=\Bscr$.
\end{lem}

\begin{proof}
Follows from Proposition \ref{prop:B is ideal w.r.to Toeplitz}. 
\end{proof}

Now we record the following properties of the collection $\Bscr$ defined 
in \eqref{eq:Iscr algebra}:

\begin{rem}\label{T phi is 0 in B} 
Suppose $\phi \in L^{\infty}(T^n)$ and $T_{\phi}\in \Bscr(\psi,\bm{0})$, where $|\psi| > 0$ a.e. on $\T^n$. Then
\[
|\widetilde{T_{\phi\psi}}(\bxi)|\leq\|T_{\phi\psi}k_{\bxi}\|\leq \|[T_{\phi},T_{\psi})k_{\bxi}\|+\|T_{\phi}T_{\psi}k_{\bxi}\|
\to 0\quad\text{radially.}
\]
Now Theorem \ref{thm:radial limit of Toeplitz operaor} yields that $\phi\psi \equiv 0$ a.e. on $\T^n$. But $|\psi| > 0$ a.e. on $\T^n$, hence $\phi\equiv 0$ a.e. on $\T^n$.
Thus, any Toeplitz operator in $\Bscr(\psi,\bm{0})$ corresponds to the symbol $0$.		
\end{rem}

For $\psi\in L^{\infty}(\T^n)$ with $|\psi| =1$ a.e. on $\T^n$, consider
\[
\Ascr_{\psi}:=\{T_{\phi}+A : \phi\in L^{\infty}(\T^n), A\in \Bscr(\psi,\bm{0})\}.
\]

\begin{thm}\label{thm:extended symbol map theorem}
The collection $\Ascr_{\psi}$ satisfies the following:
\begin{enumerate}[label={\textnormal{(\arabic*)}}]
\item \label{thm:extension to A (1)}
If $T\in\Ascr_{\psi}$, then $\widetilde{T}\to \phi$ radially for some $\phi\in L^{\infty}(\T^n)$.
\item \label{thm:extension to A (2)} $\Ascr_{\psi}$ is a $C^*$-algebra.
\item \label{thm:extension to A (3)} Define
\[
\Sigma_{\psi}:\Ascr_{\psi} \to L^{\infty}(\T^n) \quad \text{by} \quad \Sigma_{\psi}(T)=\phi,
\]
where $\phi$ is as in \ref{thm:extension to A (1)}. Then $\Sigma_{\psi}$ induces a 
$C^*$-algebra homomorphism from $\Ascr_{\psi}$ onto $L^{\infty}(\T^n)$ such that 
$\Sigma_{\psi}(T_{\phi})=\phi$ for all $\phi\in L^{\infty}(\T^n)$.
\end{enumerate}
\end{thm}

\begin{proof}
\begin{enumerate}[label={\textnormal{(\arabic*)}}]

\item Let $T\in\Ascr_{\psi}$. Then $T=T_{\phi}+A$ for some $\phi\in L^{\infty}
(\T^n),A\in\Bscr(\psi,\bm{0})$. Now Theorem \ref{thm:radial limit of Toeplitz operaor}
implies that
\[
\widetilde{T}=\widetilde{T_{\phi}}+\widetilde{A}\to\phi \quad \text{radially}.
\]
Also
\begin{align*}
|\widetilde{A}(\bxi)| 
&= |\la Ak_{\bxi}, k_{\bxi} \ra|\\
&= |\la \psi Ak_{\bxi}, \psi k_{\bxi} \ra|\\
&\leq |\la P_{H^2(\D^n)}(\psi Ak_{\bxi}), P_{H^2(\D^n)}(\psi k_{\bxi}) \ra|+
|\la P_{H^2(\D^n)^\perp}(\psi Ak_{\bxi}), P_{H^2(\D^n)^\perp}(\psi k_{\bxi}) \ra|\\
&= |\la T_{\psi}Ak_{\bxi}, T_{\psi}k_{\bxi} \ra|+|\la H_{\psi}Ak_{\bxi}, H_{\psi}k_{\bxi} \ra|\\
&\leq \|A^*T_{\psi}^*T_{\psi}k_{\bxi}\|+\|H_{\psi}A\|\|H_{\psi}k_{\bxi}\|.
\end{align*}
Thus, Proposition \ref{prop:B is ideal w.r.to Toeplitz} and Theorem \ref{thm:Hankel radial limit 0} infer that $\widetilde{A}(\bxi)\to 0$ radially.

\item It is immediate from Lemma \ref{B is c* algebra} that $\Ascr_{\psi}$ is linear space and closed under adjoint. Let $\phi_1,\phi_2\in L^{\infty}(\T^n)$, and $A_1,A_2\in \Bscr(\psi,\bm{0})$. Then
\begin{align}
(T_{\phi_1}+A_1)(T_{\phi_2}+A_2)
&= T_{\phi_1}T_{\phi_2}+T_{\phi_1}A_2+A_1T_{\phi_2}+A_1A_2 \notag \\
&= T_{\phi_1\phi_2}-([T_{\phi_1},T_{\phi_2})-T_{\phi_1}A_2-A_1T_{\phi_2}-A_1A_2). 
\label{eq:extended symbol map thm}
\end{align}
Using Lemma \ref{B is c* algebra}, Lemma \ref{lem: B contains semicommutator}, and Proposition \ref{prop:B is ideal w.r.to Toeplitz}, we have $\Ascr_{\psi}$ is closed 
under multiplication.

To show that $\Ascr_{\psi}$ is closed, we show $\Ascr_{\psi}=\overline{\Ascr_{\psi}}$. Consider 
\[
\eta_0:\Ascr_{\psi}\to L^{\infty}(\T^n) \quad \text{defined by} \quad \eta_{0}(T_{\phi}+A)=\phi.
\]
It is easily seen that $\eta_0$ is a well-defined bounded linear map from $\Ascr_{\psi}$ onto $L^{\infty}(\T^n).$ Suppose $\eta$ is the bounded linear extension of $\eta_0$ from $\overline{\Ascr_{\psi}}$ onto $L^{\infty}(\T^n)$. Assume that $T\in \overline{\Ascr_{\psi}}$. Then there exists a sequence $\{T_{\phi_m}+A_m\}$ in 
$\Ascr_{\psi}$ such that $\|(T_{\phi_m}+A_m)-T\|\to 0$ as $m \to\infty$, and hence 
$\phi_m=\eta(T_{\phi_m}+A_m)\to \eta(T)$ in $L^{\infty}(\T^n).$ Thus 
\[
\|T_{\phi_m}-T_{\eta(T)}\|=\|\phi_m-\eta(T)\|_\infty\to 0 \quad \mbox{as}~ m\to\infty.
\]
Therefore, $A_{m}=(T_{\phi_m}+A_m)-T_{\phi_m}\to T-T_{\eta(T)}$ {as}~ $m\to\infty$. 
Since $\Bscr(\psi,\bm{0})$ is closed, $T-T_{\eta(T)}\in\Bscr(\psi,\bm{0})$. Therefore, $T=T_{\eta(T)}+(T-T_{\eta(T)})\in\Ascr_{\psi}$.

\item It is easy to check that the mapping $\Sigma_{\psi}$ is well-defined linear, preserves conjugation. Also \eqref{eq:extended symbol map thm} yields that the map is multiplicative too. Thus $\Sigma_{\psi}$ is a $C^*$-algebra homomorphism from $\Ascr_{\psi}$ onto $L^{\infty}(\T^n)$, and also $\Sigma_{\psi}(T_{\phi})=\phi$ with ker$(\Sigma_{\psi})=\Bscr(\psi,\bm{0})$.
\end{enumerate}
\end{proof}

It is obvious from Lemma \ref{lem: B contains semicommutator} that 
$\Tscr(L^{\infty}(\T^n)) \subseteq\Ascr_{\psi}$. Following example shows that Theorem \ref{thm:extended symbol map theorem} is indeed an extension of Theorem \ref{thm:douglas theorem for polydisk}.

\begin{ex}
Let $P=P_{\mathcal{M}}$ be the orthogonal projection of $H^2(\D)$ onto the closed subspace $\mathcal{M}$, where $\mathcal{M}:=\bigvee_{m=0}^{\infty}\{z^{2^m}\}$. Then for $m \geq 1$
\[
\|(T_zP-PT_z)z^{2^m}\|=\|z^{2^m+1}-Pz^{2^m+1}\|=1.
\]
But $z^{2^m}\to 0$ weakly as $m\to \infty$, and hence $T_zP-PT_Z$ is not a compact operator. Therefore $P\notin \Tscr(L^{\infty}(\T^n))$ (see \cite{Bar_Hlms-ASYMPTOTIC TOEPLITZ OPERATORS}). Also it is shown in \cite[Example 8]{Engls-SYMBOL MAP VIA BEREZIN TRAN. ON H2} that $\|Pk_{re^{i\theta}}\|\to 0$ as $r\to 1^-$ for a.e.\ $e^{i\theta}\in \T.$ Consider an operator $A$ on $H^2(\D^n)$ defined as
\[
A=\overbrace{P\otimes I\otimes\dots\otimes I}^{n\text{-times}},
\]
where $I$ denotes the identity operator on $H^2(\D)$. Now for $\bxi=(\xi_1, \dots,\xi_n)\in\D^n$,
\begin{align*}
\|Ak_{\bxi}\|
&= \|(P\otimes I\otimes \cdots \otimes I)(k_{\xi_1}\otimes k_{\xi_2}\otimes\cdots\otimes k_{\xi_n})\|\\
&= \|Pk_{\xi_1}\otimes k_{\xi_2}\otimes\cdots\otimes k_{\xi_n}\|\\
&= \|Pk_{\xi_1}\| \|k_{\xi_2}\| \cdots \|k_{\xi_n}\|\\
&= \|Pk_{\xi_1}\|.
\end{align*}
It follows that $\|Ak_{\bxi}\|\to 0$ radially, and $A\in \Ascr_{\psi}$, in fact $A\in \Bscr$. Let $T_{z_i}$ denote the multiplication operator on $H^2(\D^n)$ by the $i$-th coordinate function $z_i$ for $1\leq i\leq n$, that is, $(T_{z_i}f)(\bxi)=\xi_if(\bxi)$ for all $\bxi \in \D^n, f\in H^2(\D^n)$. Consider
\[
\mathbb{E}:=\{T\in\B(H^2(\D^n)) : T_{z_i}^*TT_{z_{i}}-T \text{ is compact for all }
1\leq i\leq n\}.
\]
Then $\mathbb{E}$ is a norm-closed algebra. Now if $T_{\phi}$ is Toeplitz operator on $H^2(\D^n)$, then $T_{z_i}^*T_{\phi}T_{z_i}=T_{\phi}$ for all $1\leq i\leq n$ (see \cite{Mji_Jdb_Srj-TOEP. AND ASYMP. TOEP. ON H2DN}). Hence $\mathbb{E}$ contains all Toeplitz operators on $H^2(\D^n)$. We show that $A \notin \Tscr(L^{\infty}(\T^n))$ by showing that $A\notin \mathbb{E}$. Now
\[
T_{z_1}^*AT_{z_1}-A=(T_z^*PT_z-P)\otimes I\otimes\cdots\otimes I,
\]
and $T_{z}^*PT_{z}-P$ is not compact. Therefore, $T_{z_1}^*AT_{z_1}-A$ is not compact 
(see \cite{Kbrs_Zan-ON COMPACTNESS OF TENSOR PROD.}) and hence $A\notin \mathbb{E}$.
\end{ex}

Denote
\[
\Ascr=\{T_{\phi}+A : \phi\in L^{\infty}(\T^n), A\in\Bscr\}.
\]
Note that if $A\in \Bscr$ then Proposition \ref{prop:B is ideal w.r.to Toeplitz} gives that $A\in \Bscr(\psi,\bm{0})$ for each $\psi\in L^{\infty}(\T^n)$, and hence 
$\Ascr\subseteq\Ascr_{\psi}$. In the next result we prove that the class $\Ascr$ is maximal onto which the \emph{symbol map} exists and $T_{\phi}\to\phi$, provided that the symbol of an operator is obtained as the radial limit of its \textit{Berezin transform}.

\begin{prop}
Let $\mathcal{A}$ be any closed $C^*$-subalgebra of $\B(H^2(\D^n))$ containing Toeplitz operators such that there exists a $C^*$-homomorphism $\mathfrak{S}$ from $\mathcal{A}$ onto $L^{\infty}(\T^n)$ with $\mathfrak{S}(T_{\phi})=\phi$ for all $\phi\in L^{\infty}(\T^n)$, and for each $T \in \mathcal{A}$, $\mathfrak{S}(T)$ is the radial limit of $\widetilde{T}$. Then $\mathcal{A}\subseteq\Ascr.$ In particular, $\Ascr_{\psi}=\Ascr$, where $\Ascr_{\psi}$ is as in Theorem \ref{thm:extended symbol map theorem}.
\end{prop}

\begin{proof}
Let $T\in\mathcal{A}$. Suppose $\mathfrak{S}(T)=\phi$ for some $\phi$ in $L^{\infty}(\T^n)$. Then $T-T_{\phi}\in ker(\mathfrak{S})$ and hence
\[
\mathcal{A}=\{T_{\phi}+A : \phi\in L^{\infty}(\T^n), A\in ker(\mathfrak{S})\}.
\]
Now if $A\in ker(\mathfrak{S})$, then $A^*A$ is so. Therefore, 
\[
\|Ak_{\bxi}\|^2=\langle Ak_{\bxi}, Ak_{\bxi} \rangle=\langle A^*Ak_{\bxi}, k_{\bxi} \rangle=\widetilde{(A^*A)}(\bxi)\to 0 \text{ radially.}
\]
In the similar fashion, it can be shown that $\|A^*k_{\bxi}\|\to 0$ radially. Thus $ker(\mathfrak{S})\subseteq \Bscr$, and hence $\mathcal{A}\subseteq\Ascr$.

Following Theorem \ref{thm:extended symbol map theorem}, in particular, we have 
$\Ascr_{\psi}=\Ascr$ for all $\psi\in L^{\infty}(\T^n)$ with $|\psi|=1$ a.e.\ on $\T^n$.
\end{proof}

We shall conclude this section with the following application of our \emph{symbol map}
result, which gives a partial answer in some sense to the question raised by Barr\'{i}a and Halmos in \cite{Bar_Hlms-ASYMPTOTIC TOEPLITZ OPERATORS}: 

\textsf{Which projections belong to the Toeplitz algebra $\Tscr(L^{\infty}(\T))$?} 

Using the \emph{symbol map}, that is, Theorem \ref{thm:douglas theorem for polydisk}, we shall show that if $P$ is a projection in $\Tscr(L^{\infty}(\T))$ which is a finite rank perturbation of a Toeplitz operator, then the rank of $P$ is finite or $I_{H^2(\D)}-P$ is finite rank, that is, co-finite. However, for $n>1$, if a projection in 
$\Tscr(L^{\infty}(\T^n))$ is compact perturbation of a Toeplitz operator, then its rank 
is either finite or co-finite. More precisely, we have the following result:

\begin{thm}
Let $P$ be an orthogonal projection in the Toeplitz algebra $\Tscr(L^{\infty}(\T^n))$.
For $n=1$, if $P=T_{\phi}+F$, where $\phi\in L^{\infty}(\T)$ and $F$ is of finite rank, then either $P$ or $I_{H^2(\D)}-P$ is of finite rank. For $n>1$, if $P=T_{\phi}+K$, where $\phi\in L^{\infty}(\T^n)$ and $K$ is compact, then either $P$ or $I_{H^2(\D^n)}-P$ is of finite rank. 
\end{thm}

\begin{proof}
Suppose $P$ is an orthogonal projection in the Toeplitz algebra $\Tscr(L^{\infty}(\T^n))$
such that $P$ is a either finite or compact perturbation of a Toeplitz operator $T_{\phi}$ for $n\geq 1$. Then Theorem \ref{thm:douglas theorem for polydisk} yields $\overline{\phi}=\phi$ and $\phi^2=\phi$ a.e.\ on $\T^n$. Hence $\phi=\chi_{E}$, where $E=\{\bzeta\in\T^n : \phi(\bzeta)=1\}$, a measurable subset of $\T^n$. 

For $n=1$: If $P=T_{\chi_{E}}+F$, where $\phi\in L^{\infty}(\T)$ and $F$ is of finite rank, then $P^2=P$ gives
\[
T_{\chi_{E}}-T^2_{\chi_{E}}=F^2-F+T_{\chi_{E}}F+FT_{\chi_{E}},
\]
and hence $T_{\chi_{E}}T_{\chi_{E^c}}=T_{\chi_{E}}-T^2_{\chi_{E}}$ is of finite rank. Therefore, either $T_{\chi_E}=0$ or $T_{\chi_{E^c}}=0$ by \cite[Corollary 2.2]{Ding-PRODUCTS OF TOEP. OPERATORS ON POLYDISK.}. Thus, either $P$ or $I_{H^2(\D)}-P$ is of finite rank.

For $n>1$: we shall get
\[
T_{\chi_{E}}-T^2_{\chi_{E}}=K^2-K+T_{\chi_{E}}K+KT_{\chi_{E}},
\]
where $K$ is a compact operator. But $T_{\chi_{E}}T_{\chi_{E^c}}=T_{\chi_{E}}-T^2_{\chi_{E}}=H_{\chi_{E}}^*H_{\chi_E}$ which is compact, and hence $H_{\chi_{E}}$ is compact. Now \cite[Corollary 5]{Cotlar_Sad-ON ABSTRACT A-A-K THEOREM.}, gives all compact Hankel operators are zero whenever $n>1$. Thus $T_{\chi_{E}}T_{\chi_{E^c}}=0$ implies either $T_{\chi_{E}}=0$ or $T_{\chi_{E^c}}=0$ by Theorem 2.1 in \cite{Ding-PRODUCTS OF TOEP. OPERATORS ON POLYDISK.}. Therefore, $P=T_{\chi_{E}}+K$ gives either $P$ or $I_{H^2(\D^n)}-P$ is compact. Since both $P$ and $I_{H^2(\D^n)}-P$ are orthogonal projections and all compact orthogonal projections are of finite rank; therefore either $P$ or $I_{H^2(\D^n)}-P$ is of finite rank.

This completes the proof.
\end{proof}

\section{Concluding Remarks}	

The Toeplitz algebra $\Tscr(L^{\infty}(\T^n))$ provides an effective way to study the Toeplitz operators as elements of a $C^*$-algebra. More specifically, Theorem \ref{thm:douglas theorem for polydisk} shows that the study of the Toeplitz algebra 
$\Tscr(L^{\infty}(\T^n))$ is largely reduces to that of the semi-commutator ideal $\Sscr$. 
Let $\theta$ be an inner function on $\D^n$, and $\clk_{\theta} = H^2(\D^n) \ominus \theta H^2(\D^n)$. Suppose $P_{\theta}$ denotes the orthogonal projection of $H^2(\D^n)$ onto 
$\clk_{\theta}$. Then
\[
P_{\theta}=I_{H^2(\D^n)}-T_{\theta}T_{\overline{\theta}}=[T_{\theta}, T_{\overline{\theta}})\in \Sscr.
\]
In fact, if $\phi=\theta_1\overline{\eta_1}$ and $\psi=\theta_2\overline{\eta_2}$, where $\theta_i, \eta_i$ are inner functions, then 
\begin{align*}
[T_{\phi}, T_{\psi})
&= T_{\phi\psi}-T_{\phi}T_{\psi}\\
&= T_{\theta_1\overline{\eta_1}\theta_2\overline{\eta_2}}-T_{\theta_1\overline{\eta_1}}T_{\theta_2\overline{\eta_2}}\\
&= T_{\overline{\eta_1}}T_{\overline{\eta_2}}T_{\theta_1}T_{\theta_2}-T_{\overline{\eta_1}}T_{\theta_1}T_{\overline{\eta_2}}T_{\theta_2}\\
&= T_{\overline{\eta_1}}T_{\overline{\eta_2}}T_{\theta_1}(I-T_{\eta_2}T_{\overline{\eta_2}})T_{\theta_2}\\
&= T_{\overline{\eta_1}}T_{\overline{\eta_2}}T_{\theta_1}P_{\eta_2}T_{\theta_2}.
\end{align*}
In the case for $n=1$, the semi-commutator ideal coincides with the commutator ideal in 
$\Tscr(L^{\infty}(\T))$ due to the availability of Douglas-Rudin theorm; again from the above calculation the semi-commutator ideal coincides with the closed ideal in 
$\Tscr(L^{\infty}(\T))$ generated by all orthogonal projections on the model space. 
Thus, Douglas-Rudin theorem (see \cite{Dgl_Rudn-APPROXIMATION BY INNER FUNCTIONS}) plays an important role in this study. We now conclude this paper by raising some questions which naturally arise in our paper: 
	
\vspace{0.2cm}

\textsf{Question 1. Is it true that the class $\{\overline{\phi}\psi : \phi, \psi\in H^{\infty}(\T^n)\}$ is norm-dense in $L^{\infty}(\T^n)$ for $n>1$?}

\vspace{0.2cm}	
	
\textsf{Question 2. Does the semi-commutator ideal coincide with the commutator ideal in 
$\Tscr(L^{\infty}(\T^n))$ for $n>1$?}

It is noteworthy to point out that the affirmative answer to {Question 1} gives an affirmative answer to {Question 2}, which we shall try to carry out in the future paper.

\vspace{0.2 cm}	
\NI\textit{Data availability:}
Data sharing is not applicable to this article as no data sets
were generated or analysed during the current study.
\vspace{0.2cm}

\NI\textit{Declarations}
\vspace{0.2cm}

\NI\textit{Conflict of interest:}
The authors have no competing interests to declare.


\end{document}